\documentclass[reqno, 12pt]{amsart}
\usepackage{amssymb,amsmath,amsfonts}
\usepackage[margin=1in]{geometry}
\usepackage{dsfont}
\usepackage{color}
\usepackage{graphicx}
\usepackage{latexsym}
\usepackage{amsmath}
\usepackage{amssymb}
\usepackage{graphics}
\usepackage[dvips]{epsfig}
\usepackage{mathrsfs}
\usepackage{mathtools}
\usepackage{leqno}
\usepackage{stmaryrd,cite}
\usepackage{txfonts}
\usepackage{cases}
\usepackage{enumerate}
\usepackage[usenames,dvipsnames,svgnames]{xcolor}
\definecolor{refkey}{rgb}{1,0,0.5}
\definecolor{labelkey}{rgb}{0,0.4,1}
\usepackage{todonotes}
\usepackage{marginnote}

\presetkeys{todonotes}{fancyline, color=green!40}{}

\makeatletter
\renewcommand{\@todonotes@drawMarginNoteWithLine}{%
	\begin{tikzpicture}[remember picture, overlay, baseline=-0.75ex]%
	\node [coordinate] (inText) {};%
	\end{tikzpicture}%
	\marginnote[{
		\@todonotes@drawMarginNote%
		\@todonotes@drawLineToLeftMargin%
	}]{
		\@todonotes@drawMarginNote%
		\@todonotes@drawLineToRightMargin%
	}%
}
\makeatother
\usepackage{listings}
\usepackage{ifpdf}
\ifpdf
\usepackage[CJKbookmarks=true,
         hyperindex=true,
         pdfstartview=FitH,
         bookmarksnumbered=true,
         bookmarksopen=true,
         colorlinks=true,
         citecolor=blue,
         linkcolor=blue,
         urlcolor=blue,
         pdfborder=001,
         pdfauthor={},
         pdftitle={},
         pdfkeywords={},
         ]{hyperref}
\else
\usepackage[
         hypertex,
         hyperindex,
         linkcolor=blue,
         unicode,
         citecolor=blue%
         ]{hyperref}
\fi
\usepackage{cleveref}

\allowdisplaybreaks

\numberwithin{equation}{section}
\newtheorem{theorem}{Theorem}[section]

\newtheorem{lemma}[theorem]{Lemma}

\newtheorem{remark}[theorem]{Remark}
\newtheorem{definition}[theorem]{Definition}

\newcommand{\be}{\begin{equation}}
\newcommand{\ee}{\end{equation}}

\newcommand{\bee}{\begin{equation*}}
\newcommand{\eee}{\end{equation*}}

\newcommand{\bse}{\begin{subequations}}
\newcommand{\ese}{\end{subequations}}

\newcommand{\bs}{\begin{split}}
\newcommand{\es}{\end{split}}

\begin{document}

\author{Hairong Liu$^{1}$}\thanks{$^{1}$School of Mathematics and Statistics, Nanjing University of Science and Technology,  Nanjing 210094, China.
E-mail: hrliu@njust.edu.cn}

\author{Long Tian$^{2}$}\thanks{$^{2}$School of Mathematics and Statistics, Nanjing University of Science and Technology,  Nanjing 210094, China.
E-mail: tianlong19850812@163.com}

\author{Xiaoping Yang$^{3}$}\thanks{$^{3}$School of Mathematics,  Nanjing University, Nanjing 210093, China.
E-mail: xpyang@nju.edu.cn}

\title[]{Measure upper bounds of nodal sets of solutions to  Dirichlet problem of Schr\"{o}dinger equations}

\begin{abstract}

 In this paper, we focus on estimating measure upper bounds of nodal sets of solutions to the following boundary value problem
\begin{equation*}
\left\{
\begin{array}{lll}
\Delta u+Vu=0\quad \mbox{in}\ \Omega,\\[2mm]
u=0\quad \mbox{on}\ \partial\Omega,
\end{array}\right.
\end{equation*}
where  $V\in W^{1,\infty}(\Omega)$ is  a potential function,  and
$\Omega \subset \mathbb{R}^n$ ($n \geq 2$) is a bounded domain whose boundary is of class  $C^{1,\alpha}$ for any $0<\alpha<1$.
By developing a delicate dividing iteration procedure, we show that  upper bound  of the $(n-1)$-dimensional
Hausdorff measure of the nodal set
of $u$ in $\Omega$ is
$$C\Big(1+\log\left(\|\nabla V\|_{L^{\infty}(\Omega)}+1\right)\Big)\cdot\left(\|V\|_{L^{\infty}(\Omega)}^{\frac{1}{2}}+\|\nabla V\|_{L^{\infty}(\Omega)}^{\frac{1}{2}}+1\right),$$ provided  $V$ is analytic,
here $C$ is  a positive constant depending only on $n$ and $\Omega$. In particular, if $\|\nabla V\|_{L^{\infty}(\Omega)}$ is small,  the upper bound for the measure of the  nodal set of $u$ is $C\left(\|V\|^{\frac{1}{2}}_{L^{\infty}(\Omega)}+1\right)$, which is sharp in the sense of a famous conjecture of Yau.



\end{abstract}

\maketitle

\section{Introduction and main theorem}

Let $u$ be the Laplacian eigenfunctions
\begin{equation*}
-\Delta_{\mathcal{M}} u=\lambda u
\end{equation*}
on a compact $C^{\infty}$-smooth Riemannian manifold  $\mathcal{M}$ without boundary, S.T. Yau \cite{Yau} conjectured that the Hausdorff
measure of nodal sets can be controlled by eigenvalues as
\begin{equation*}
C_1\sqrt{\lambda}\leq \mathcal{H}^{n-1}(\left\{x\in\mathcal{M}\ |\ u(x)=0\}\right)\leq C_2\sqrt{\lambda},
\end{equation*}
where the positive constants
$C_1$ and $C_2$  depend on the manifold $\mathcal{M}$,  $ \mathcal{H}^{n-1}$ denotes the $(n-1)$-dimensional Hausdorff measure.

The investigation of nodal sets has made significant progress in different geometric  and PDE environments.
For real analytic manifolds,  Yau's conjecture was first solved in the two-dimensional case through independent work by Br\"uning and Yau, while Donnelly and Fefferman \cite{Donnelly} established the higher dimensional case. Lin \cite{Lin} developed an alternative approach to obtain upper bounds for general second-order elliptic equations on analytic manifolds.

In the smooth manifolds, Donnelly and Fefferman \cite{Donnelly2} proved that the measure upper bound of nodal sets of eigenfunctions is $C\lambda^{\frac{3}{4}}$ in 2-dimensional case.   This was extended to arbitrary dimensions by Hardt and Simon \cite{Hardt}, who obtained the bound $\lambda^{C\sqrt{\lambda}}$.
Logunov and Malinnikova \cite{LM} improved the two dimensional  result in \cite{Donnelly2} to $C\lambda^{\frac{3}{4}-\epsilon}$ for some $\epsilon\in (0,\frac{1}{4})$. Recently, Logunov \cite{Logunov} showed the general upper bound $\lambda^{\alpha}$ for some $\alpha>\frac{1}{2}$ for the higher dimensions.

Regarding the lower bound of Yau's conjecture, substantial progress was made through a series of important works, culminating in Logunov's complete resolution \cite{Logunov3} (see also references therein for related results). In the context of manifolds with boundary, Donnelly and Fefferman \cite{DF1990} established both upper and lower bounds for nodal sets of Laplace-Beltrami eigenfunctions on compact, connected, real-analytic manifolds with analytic boundary. The smooth case was addressed by Ariturk \cite{Ar2013}, who obtained lower bounds for  Dirichlet and Neumann eigenfunctions on compact $C^{\infty}$-smooth Riemannian manifolds with boundary. Recently, Logunov, Malinnikova, Nadirashvili
and Nazarov \cite{Logunov2}
obtained  the sharp upper bound for the area of the nodal sets of Dirichlet Laplace eigenfunctions
for the case of bounded domains with $C^1$-smooth boundary.
 For extensions to higher-order operators with various boundary conditions, we refer to the works \cite{K1995, lin2022, TY2}.

The goal of the paper is to investigate and give the measure upper bound of nodal sets of  non-trivial  solutions
to  the following  Dirichlet problem of Schr\"{o}dinger equations:
\begin{equation}\label{basic equation}
\left\{
\begin{array}{lll}
\Delta u+Vu=0\quad \mbox{in}\ \Omega,\\
u=0\quad \mbox{on}\ \partial\Omega.
\end{array}
\right.
\end{equation}


Our main result is the following theorem.
\begin{theorem}\label{main result 2}
Assume $\Omega\subset\mathbb{R}^{n}$ is a bounded domain with $C^{1,\alpha}$ (for any $0<\alpha<1$) boundary and $V\in W^{1,\infty}(\Omega)$ is
  analytic.
Let $u$ be a non-trivial solution of (\ref{basic equation}).
 Then
\begin{equation*}
\mathcal{H}^{n-1}\Big(\{x\in\Omega\ |\ u(x)=0\}\Big)\leq C\Big(1+\log\left(\|\nabla V\|_{L^{\infty}}+1\right)\Big)\cdot\left(\|V\|_{L^{\infty}}^{\frac{1}{2}}+\|\nabla V\|_{L^{\infty}}^{\frac{1}{2}}\right),
\end{equation*}
where $C$ is a positive constant depending only on $n$ and $\Omega$.
In particular, if $\|\nabla V\|_{L^{\infty}}$ is small, then
\begin{equation*}
\mathcal{H}^{n-1}\Big(\{x\in\Omega\ |\ u(x)=0\}\Big)\leq C\|V\|^{\frac{1}{2}}_{L^{\infty}}.
\end{equation*}
\end{theorem}

The analysis of nodal sets of solutions to  the Dirichlet boundary value problem (\ref{basic equation})  with non-constant potentials $V(x)$ in $C^{1,\alpha}$ domain brings about  two essential difficulties. First, the presence of a non-constant potential term   prevents any direct reduction to harmonic function techniques,
 a method commonly employed in the study of nodal sets to eigenfunctions (see [28]).
Second, in the case of the non-analyticity of the domain, the
 powerful tools of analytic continuation  become unavailable. These combined features necessitate the development of new quantitative approaches that can simultaneously handle the non-harmonic nature of the solutions and the boundary behavior in low-regularity settings.

In order to overcome the difficulties,
we first lift the original equation to that with a nonpositive potential,
and then define a variant frequency function and a doubling index of the solution to the lifting equation.  More precisely,
 let $u$ solve the original problem (\ref{basic equation}).
  Define
\begin{equation*}
\bar{u}(x,x_{n+1})=\bar{u}(z)=u(x)e^{\sqrt{\lambda}x_{n+1}},
\end{equation*}
where $\lambda=\|V\|_{L^{\infty}}+\|\nabla V\|_{L^{\infty}}$,
then for any $R>1$, $\bar{u}$ satisfies the equation
\begin{equation}\label{add one dim-1}
\left\{
\begin{array}{lll}
\Delta_{x,t}\bar{u}(x,t)+\bar{V}(x)\bar{u}(x,t)=0,\quad (x,t)\in\Omega\times(-R,R)\equiv\widetilde{\Omega},\\[2mm]
\bar{u}|_{\partial\Omega\times(-R,R)}=0,
\end{array}
\right.
\end{equation}
where $\bar{V}(x)=V(x)-\lambda\in W^{1,\infty}\left(\widetilde{\Omega}\right)$
and   $\bar{V}\leq 0$, $\bar{V}+|\nabla\bar{V}|\leq 0$ in $\widetilde{\Omega}$.
By establishing a monotonicity formula for the frequency function and deriving precise a priori estimates, we obtain quantitative upper bounds for the doubling index that depend explicitly on the potential through the quantity $\|V\|^{1/2}_{W^{1,\infty}}$ (see Lemma \ref{upper bound of frequency}).

The Key advantage of lifting technique is to obtain the non-positivity of the new potential $\bar{V}$ in (\ref{add one dim-1}) which enables to establish an useful optimal estimate of the doubling index and an improved $L^{\infty}-L^{2}$ estimate through De Giorgi-Nash-Moser method. These, in turn, are able to  facilitate to derive the  upper bounds of the doubling index,  which have the explicit expression as $\|V\|^{\frac{1}{2}}_{W^{1,\infty}}$.

The measure upper bound estimate is much challenging. We
develop a refined iterative procedure to estimate  the measure upper bounds
 for the nodal sets of solutions through careful controlling doubling indexes.
More precisely, we first divide the domain into two regions: one is
away from the boundary and another is near the boundary. In the interior,
 we combine the lifting technique with complexification methods and integral geometry formula to derive measure bounds that is controlled by  $\|V\|^{1/2}_{W^{1,\infty}}$ (Lemmas \ref{interior nodal set}-\ref{nodal set of cube}).
 In order to deal with the nodal sets  near the boundary, we introduce a novel iterative partitioning procedure inspired by \cite{Logunov3}, but incorporating several elaborate techniques to control the number of cubes with large doubling indexes and estimate the measures of nodal sets in  cubes with different doubling indexes approximating the boundary $\partial\Omega$.
Actually, the procedure involves: (i) establishing propagation of smallness and successively partitioning a given cube; (ii) carefully estimating the number of  cubes with large doubling indexes; and (iii) computing the measure of nodal sets in a given cube by using different arguments to deal with the cubes having small or  large doubling indexes.

The rest of the paper is organized as follows. In Section 2, we first lift the original equation to that with a signed potential, then
define a frequency function and
prove the monotonicity formula and doubling estimates for the corresponding frequency. In Section 3, we define the doubling index, estimate the upper bound of the doubling index and prove that the maximal vanishing order of $u$ in $\Omega$ is  $C\left(\|V\|^{\frac{1}{2}}_{W^{1,\infty}(\Omega)}+1\right)$. Finally, in Section 4, we estimate the upper bound of the measure of nodal sets of $u$ in $\Omega$ and prove Theorem \ref{main result 2}.

\section{Frequency function}

In this section, we first transform the equation  (\ref{basic equation}) by lifting into an equation with a signed potential term, then define a frequency function $N(r)$ and establish its  monotonicity. Finally, with help of the  monotonicity of $N(r)$, we derive some doubling estimates and a changing center property of $N(r)$.

Let $u=u(x)$ be a solution to (\ref{basic equation}) and the lift of $u$ be defined by
\begin{align*}
\bar{u}(x,t)=u(x)e^{\sqrt{\lambda}t},
\end{align*}
where \begin{equation}\label{definition of lambda}
\lambda=\|V\|_{W^{1,\infty}}\equiv\|V\|_{L^{\infty}}+\|\nabla V\|_{L^{\infty}}.
 \end{equation}
 Then for any $R>1$, $\bar{u}$ satisfies the equation
\begin{equation}\label{add one dim}
\left\{
\begin{array}{lll}
\Delta_{x,t}\bar{u}(x,t)+\bar{V}(x)\bar{u}(x,t)=0,\quad (x,t)\in\Omega\times(-R,R)\equiv\widetilde{\Omega},\\[2mm]
\bar{u}|_{\partial\Omega\times(-R,R)}=0,
\end{array}
\right.
\end{equation}
where $\bar{V}(x)=V(x)-\lambda\in W^{1,\infty}\left(\widetilde{\Omega}\right)$
and   $\bar{V}\leq 0$, $\bar{V}+|\nabla\bar{V}|\leq 0$ in $\widetilde{\Omega}$  thanks to the definition of  $\lambda$. Noting that the nodal sets of $u$ and $\bar{u}$ are mutually controlled,
from now on, we only need to study equation (\ref{add one dim}) and $\bar{u}(x,t)$. Let $z=(x,t)$,
and $B_r(z_0)$ denote the ball centered at $z_0$ with radius $r$  in $\mathbb{R}^{n+1}$.

It is worth noting that the fact $\bar{V}\leq 0$ enables us to obtain the optimal estimate of the doubling index (Lemma 3.5).  Moveover,
in view of the fact  $\bar{V}\leq0$, we  establish an improved $L^{\infty}-L^2$ estimate (through De Giorgi-Nash-Moser) where the constant independent of $\bar{V}$ (Lemma \ref{degiorgi}).
In turn,
all constants $C$ appearing in our paper are independent of $V$. This is the key advantage of the technique for lifting the equation (\ref{basic equation}) into (\ref{add one dim}).

We first define a variant frequency function which goes back to Kukavica \cite{Kukavica}
for Ginzburg-Landau equations.
\begin{definition}
For $z_0=(x_0,0)$ with $x_0\in\Omega$, let
\begin{align}\label{H}
H_{\bar{u}}(z_0,r)=\int_{B_r(z_0)\cap\widetilde{\Omega}}\bar{u}^2dz,
\end{align}
\begin{align*}
I_{\bar{u}}(z_0,r)=\int_{B_r(z_0)\cap\widetilde{\Omega}}|\nabla\bar{u}|^2(r^2-|z-z_0|^2)dz
-\int_{B_r(z_0)\cap\widetilde{\Omega}}\bar{V}\bar{u}^2(r^2-|z-z_0|^2)dz.
\end{align*}
The frequency function is defined as
\begin{equation*}
N_{\bar{u}}(z_0,r)=\frac{I_{\bar{u}}(z_0,r)}{H_{\bar{u}}(z_0,r)}.
\end{equation*}
\end{definition}

We note that, by using equation (\ref{add one dim}), it is easy to see that
\begin{align}\label{another-I}
I_{\bar{u}}(z_0,r)=2\int_{B_r(z_0)\cap\widetilde{\Omega}}\bar{u}\nabla\bar{u}\cdot(z-z_0)dz.
\end{align}
 Indeed,   integrating by parts and using equation (\ref{add one dim}), one has
\begin{align*}
&2\int_{B_r(z_0)\cap\widetilde{\Omega}}\bar{u}\nabla\bar{u}\cdot(z-z_0)dz\\[2mm]
&=-\int_{B_r(z_0)\cap\widetilde{\Omega}}\bar{u}\nabla\bar{u}\cdot\nabla(r^2-|z-z_0|^2)dz\\[2mm]
&=\int_{B_r(z_0)\cap\widetilde{\Omega}}\mbox{div}(\bar{u}\nabla\bar{u})(r^2-|z-z_0|^2)dz\\[2mm]
&=\int_{B_r(z_0)\cap\widetilde{\Omega}}|\nabla\bar{u}|^2(r^2-|z-z_0|^2)dz
+\int_{B_r(z_0)\cap\widetilde{\Omega}}\bar{u}\Delta\bar{u}(r^2-|z-z_0|^2)dz\\[2mm]
&=I_{\bar{u}}(z_0,r).
\end{align*}

Now, we will establish the monotonicity property of the
frequency function, which is the key tool to prove our main theorem. Precisely
\begin{lemma}\label{monotonicity}
Let $\bar{u}\in W^{1,2}(\widetilde{\Omega})$ be  a solution of (\ref{add one dim}), then  for   $z_0=(x_0,0)$ with $x_0\in\Omega$,
$0<r<1$, if $B_r(x_0)\cap\Omega$ is  star-shaped with respect to $x_0$, it holds
\begin{equation*}
\frac{d}{dr}N_{\bar{u}}(z_0,r)\geq0.
\end{equation*}
\end{lemma}

\begin{proof}
Firstly, we note that
\begin{equation}\label{cond}
\bar{u}|_{\partial{\widetilde{\Omega}}\cap B_r(z_0)}=0
\end{equation}  in virtue of $\bar{u}|_{\partial\Omega\times(-R,R)}=0$,  the definition of $\widetilde{\Omega}$ and $r<1$.
Taking the derivative for $H_{\bar{u}}(z_0,r)$ with respect to $r$,  and using the boundary condition (\ref{cond}), one has
\begin{align}\label{derivative of H}
H_{\bar{u}}'(z_0,r)&=\int_{\partial B_r(z_0)\cap\widetilde{\Omega}}\bar{u}^2d\sigma\nonumber\\[2mm]
&=\frac{1}{r}\int_{\partial \left(B_r(z_0)\cap\widetilde{\Omega}\right)}\bar{u}^2(z-z_0)\cdot\nu d\sigma\nonumber\\[2mm]
&=\frac{1}{r}\int_{B_r(z_0)\cap\widetilde{\Omega}}\mbox{div}(\bar{u}^2(z-z_0)) dz\nonumber\\[2mm]
&=\frac{n+1}{r}H_{\bar{u}}(z_0,r)+\frac{2}{r}\int_{B_r(z_0)\cap\widetilde{\Omega}}\bar{u}\nabla\bar{u}\cdot (z-z_0)dz,
\end{align}
where $\nu$ is the unit outer normal vector on $\partial\left(B_r(z_0)\cap\widetilde{\Omega}\right)$.
Putting (\ref{another-I})  into (\ref{derivative of H}), one concludes
\begin{equation}\label{H-fin}
H_{\bar{u}}'(z_0,r)=\frac{n+1}{r}H_{\bar{u}}(z_0,r)+\frac{1}{r}I_{\bar{u}}(z_0,r).
\end{equation}
Now we calculate the derivative of $I_{\bar{u}}(z_0,r)$. Recalling
\begin{align*}
I_{\bar{u}}(z_0,r)&=\int_{B_r(z_0)\cap\widetilde{\Omega}}|\nabla\bar{u}|^2(r^2-|z-z_0|^2)dz
-\int_{B_r(z_0)\cap\widetilde{\Omega}}\bar{V}\bar{u}^2(r^2-|z-z_0|^2)dz\\[2mm]
&\equiv I^{(1)}_{\bar{u}}(z_0,r)-I^{(2)}_{\bar{u}}(z_0,r).
\end{align*}
Firstly,
using the following identity,
\begin{equation*}
2|z-z_0|^2=-(z-z_0)\cdot\nabla(r^2-|z-z_0|^2)
\end{equation*}
and integrating by parts, it holds
\begin{align}\label{I1}
\frac{d}{dr}I^{(1)}_{\bar{u}}(z_0,r)&=2r\int_{B_r(z_0)\cap\widetilde{\Omega}}|\nabla\bar{u}|^2dz
\nonumber\\[2mm]
&=\frac{2}{r}\int_{B_r(z_0)\cap\widetilde\Omega}|\nabla\bar{u}|^2(r^2-|z-z_0|^2)dz
+\frac{2}{r}\int_{B_r(z_0)\cap\widetilde{\Omega}}|\nabla\bar{u}|^2|z-z_0|^2dz
\nonumber\\[2mm]
&=\frac{2}{r}I_{\bar{u}}^{(1)}(z_0,r)
-\frac{1}{r}\int_{B_r(z_0)\cap\widetilde{\Omega}}|\nabla\bar{u}|^2(z-z_0)\cdot\nabla(r^2-|z-z_0|^2)dz
\nonumber\\[2mm]
&=\frac{2}{r}I_{\bar{u}}^{(1)}(z_0,r)-\frac{1}{r}\int_{B_r(z_0)\cap\partial\widetilde{\Omega}}|\nabla\bar{u}|^2(r^2-|z-z_0|^2)(z-z_0)\cdot\nu d\sigma
\nonumber\\[2mm]
&+\frac{1}{r}\int_{B_r(z_0)\cap\widetilde{\Omega}}\mbox{div}(|\nabla\bar{u}|^2(z-z_0))(r^2-|z-z_0|^2)dz
\nonumber\\[2mm]
&=:\frac{2}{r}I^{(1)}_{\bar{u}}(z_0,r)-\frac{1}{r}J_1+\frac{1}{r}J_2.
\end{align}
Integrating by parts again, and using the equation (\ref{add one dim}), the term $J_2$  becomes
\begin{align}\label{above}
J_2&=(n+1)\int_{B_r(z_0)\cap\widetilde{\Omega}}|\nabla\bar{u}|^2(r^2-|z-z_0|^2)dz+2\int_{B_r(z_0)\cap\widetilde{\Omega}}\nabla\bar{u}\cdot\nabla^2\bar{u}\cdot(z-z_0)(r^2-|z-z_0|^2)dz
\nonumber\\[2mm]
&=(n+1)I^{(1)}_{\bar{u}}(z_0,r)+2\int_{\partial(B_r(z_0)\cap\widetilde{\Omega})}\left(\nabla\bar{u}\cdot\nu\right)\left(\nabla\bar{u}\cdot(z-z_0)\right)(r^2-|z-z_0|^2)d\sigma
\nonumber\\[2mm]
&-2\int_{B_r(z_0)\cap\widetilde{\Omega}}\nabla\bar{u}\cdot(z-z_0)\Delta\bar{u}(r^2-|z-z_0|^2)dz
-2\int_{B_r(z_0)\cap\widetilde{\Omega}}|\nabla\bar{u}|^2(r^2-|z-z_0|^2)dz
\nonumber\\[2mm]
&+4\int_{B_r(z_0)\cap\widetilde{\Omega}}\left(\nabla\bar{u}\cdot(z-z_0)\right)^2dz\nonumber\\[2mm]
&=(n-1)I^{(1)}_{\bar{u}}(z_0,r)+2\int_{B_r(z_0)\cap\widetilde{\Omega}}\bar{V}\nabla\bar{u}\cdot(z-z_0)\bar{u}(r^2-|z-z_0|^2)dz
+4\int_{B_r(z_0)\cap\widetilde{\Omega}}(\nabla\bar{u}\cdot(z-z_0))^2dz\nonumber\\[2mm]
&+2\int_{\partial(B_r(z_0)\cap\widetilde{\Omega})}\left(\nabla\bar{u}\cdot\nu\right)\left(\nabla\bar{u}\cdot(z-z_0)\right)(r^2-|z-z_0|^2)d\sigma.
\end{align}
Since
\begin{equation*}
\nabla\bar{u}=\pm|\nabla\bar{u}|\nu  \quad \mbox{on}\quad  \partial\widetilde{\Omega}
\end{equation*}
thanks to  $\bar{u}|_{\partial\Omega\times(-R,R)}=0$. Then on $\partial\widetilde{\Omega}$ it holds
\begin{equation*}
\nabla\bar{u}\cdot\nu=\pm|\nabla \bar{u}|, \quad \nabla\bar{u}\cdot(z-z_0)=\pm|\nabla\bar{u}|(z-z_0)\cdot\nu.
\end{equation*}
Therefore, the boundary integral term in the right hand side of (\ref{above}) becomes
\begin{align}\label{1}
&\int_{\partial(B_r(z_0)\cap\widetilde{\Omega})}\left(\nabla\bar{u}\cdot\nu\right)\left(\nabla\bar{u}\cdot(z-z_0)\right)(r^2-|z-z_0|^2)d\sigma\nonumber\\[2mm]
&=\int_{B_r(z_0)\cap\partial\widetilde{\Omega}}\left(\nabla\bar{u}\cdot\nu\right)\left(\nabla\bar{u}\cdot(z-z_0)\right)(r^2-|z-z_0|^2)d\sigma\nonumber\\[2mm]
&=\int_{B_r(z_0)\cap\partial\widetilde{\Omega}}|\nabla\bar{u}|^2(r^2-|z-z_0|^2)(z-z_0)\cdot\nu d\sigma=J_1.
\end{align}
Putting (\ref{1}) into (\ref{above})  implies  that
\begin{align*}
J_2&=(n-1)I^{(1)}_{\bar{u}}(z_0,r)+2J_1\\[2mm]
&+2\int_{B_r(z_0)\cap\widetilde{\Omega}}\bar{V}\nabla\bar{u}\cdot(z-z_0)\bar{u}(r^2-|z-z_0|^2)dz
+4\int_{B_r(z_0)\cap\widetilde{\Omega}}(\nabla\bar{u}\cdot(z-z_0))^2dz.
\end{align*}
Plugging $J_1$ and $J_2$ into (\ref{I1}), one obtains
\begin{align*}
\frac{d}{dr}I^{(1)}_{\bar{u}}(z_0,r)&=\frac{n+1}{r}I^{(1)}_{\bar{u}}(z_0,r)+\frac{1}{r}\int_{B_r(z_0)\cap\partial\widetilde{\Omega}}|\nabla\bar{u}|^2(r^2-|z-z_0|^2)(z-z_0)\cdot\nu d\sigma\\[2mm]
&+\frac{2}{r}\int_{B_r(z_0)\cap\widetilde{\Omega}}\bar{V}\nabla\bar{u}\cdot(z-z_0)\bar{u}(r^2-|z-z_0|^2)dz
+\frac{4}{r}\int_{B_r(z_0)\cap\widetilde{\Omega}}(\nabla\bar{u}\cdot(z-z_0))^2dz\\[2mm]
&\equiv
\frac{n+1}{r}I^{(1)}_{\bar{u}}(z_0,r)+\frac{1}{r}J_1+\frac{2}{r}J_3+\frac{4}{r}J_4.
\end{align*}
Since $B_r(x_0)\cap\Omega$ is star-shaped with respect to $x_0$, so $B_r(z_0)\cap\widetilde{\Omega}$ is star-shaped with respect to $z_0$.  Then $\nu \cdot (z-z_0) \geq 0$ for any $z\in \partial\widetilde{\Omega}\cap B_r(z_0)$, that is
$J_1\equiv\int_{B_r(z_0)\cap\partial\widetilde{\Omega}}|\nabla\bar{u}|^2(r^2-|z-z_0|^2)(z-z_0)\cdot\nu d\sigma\geq 0$.
Thus
\begin{equation}\label{calculation of I1}
\frac{d}{dr}I^{(1)}_{\bar{u}}(z_0,r)\geq\frac{n+1}{r}I^{(1)}_{\bar{u}}(z_0,r)+\frac{2}{r}J_3+\frac{4}{r}J_4.
\end{equation}
Next, we calculate the derivative of $I_{\bar{u}}^{(2)}(z_0,r)$,
\begin{align*}
\frac{d}{dr}I^{(2)}_{\bar{u}}(z_0,r)&=2r\int_{B_r(z_0)\cap\widetilde{\Omega}}\bar{V}\bar{u}^2dz
\\[2mm]
&=\frac{2}{r}I^{(2)}_{\bar{u}}(z_0,r)+\frac{2}{r}\int_{B_r(z_0)\cap\widetilde{\Omega}}\bar{V}\bar{u}^2|z-z_0|^2dz
\\[2mm]
&=\frac{2}{r}I^{(2)}_{\bar{u}}(z_0,r)-\frac{1}{r}\int_{B_r(z_0)\cap\widetilde{\Omega}}\bar{V}\bar{u}^2(z-z_0)\cdot\nabla(r^2-|z-z_0|^2)dz
\\[2mm]
&=\frac{2}{r}I^{(2)}_{\bar{u}}(z_0,r)+\frac{1}{r}\int_{B_r(z_0)\cap\widetilde{\Omega}}div(\bar{V}\bar{u}^2(z-z_0))(r^2-|z-z_0|^2)dz
\\[2mm]
&=\frac{n+3}{r}I^{(2)}_{\bar{u}}(z_0,r)+\frac{1}{r}\int_{B_r(z_0)\cap\widetilde{\Omega}}\nabla\bar{V}\cdot(z-z_0)\bar{u}^2(r^2-|z-z_0|^2)dz
\\[2mm]
&+\frac{2}{r}\int_{B_r(z_0)\cap\widetilde{\Omega}}\bar{V}\bar{u}\nabla\bar{u}\cdot(z-z_0)(r^2-|z-z_0|^2)dz
\\[2mm]
&=\frac{n+3}{r}I^{(2)}_{\bar{u}}(z_0,r)+\frac{1}{r}J_5+\frac{2}{r}J_3,
\end{align*}
which together with (\ref{calculation of I1})  yields
\begin{align*}
\frac{d}{dr}I_{\bar{u}}(z_0,r)&=\frac{d}{dr}I^{(1)}_{\bar{u}}(z_0,r)-\frac{d}{dr}I^{(2)}_{\bar{u}}(z_0,r)\nonumber
\\&\geq\frac{n+1}{r}I_{\bar{u}}(z_0,r)+\frac{4}{r}J_4-\frac{1}{r}\Big(2I^{(2)}_{\bar{u}}(z_0,r)+J_5\Big).
\end{align*}
Since $\bar{V}=V-\lambda=-\|V\|_{L^{\infty}}-\|\nabla V\|_{L^{\infty}}+V\leq-\|\nabla V\|_{L^{\infty}}$, and $|z-z_0|\leq r\leq 1$, there holds
\begin{align*}
2I^{(2)}_{\bar{u}}(z_0,r)+J_5=\int_{B_r(z_0)\cap\widetilde{\Omega}}(2\bar{V}+\nabla V\cdot(z-z_0))\bar{u}^2(r^2-|z-z_0|^2)dz
\leq0.
\end{align*}
Therefore,
\begin{equation}\label{I-fin}
\frac{d}{dr}I_{\bar{u}}(z_0,r)\geq \frac{n+1}{r}I_{\bar{u}}(z_0,r)+\frac{4}{r}J_4.
\end{equation}
Combining (\ref{H-fin}) with (\ref{I-fin}), we conclude,
\begin{align*}
&\frac{d}{dr}N_{\bar{u}}(z_0,r)=\frac{1}{H^2_{\bar{u}}(z_0,r)}\left(\frac{d}{dr}I_{\bar{u}}(z_0,r)H_{\bar{u}}(z_0,r)-\frac{d}{dr}H_{\bar{u}}(z_0,r)I_{\bar{u}}(z_0,r)\right)\nonumber
\\[2mm]
&\quad\geq\frac{1}{H^2_{\bar{u}}(z_0,r)}\left\{\left(\frac{n+1}{r}I_{\bar{u}}(z_0,r)+\frac{4}{r}J_4\right)H_{\bar{u}}(z_0,r)-\left(\frac{n+1}{r}H_{\bar{u}}(z_0,r)+\frac{1}{r}I_{\bar{u}}(z_0,r)\right)I_{\bar{u}}(z_0,r)\right\}\nonumber
\\[2mm]
&\quad=\frac{1}{H^2_{\bar{u}}(z_0,r)}\left(\frac{4}{r}J_4H_{\bar{u}}(z_0,r)-\frac{1}{r}I^2_{\bar{u}}(z_0,r)\right).
\end{align*}
In virtue of the expression of $J_4$,  $H_{\bar{u}}(z_0,r)$ in (\ref{H}) and $I_{\bar{u}}(z_0,r)$ in (\ref{another-I}), we obtain
\begin{align*}
&\frac{4}{r}J_4H_{\bar{u}}(z_0,r)-\frac{1}{r}I^2_{\bar{u}}(z_0,r)
\\[2mm]
&=\frac{4}{r}\int_{B_r(z_0)\cap\widetilde{\Omega}}(\nabla\bar{u}\cdot(z-z_0))^2dz\int_{B_r(z_0)\cap\widetilde{\Omega}}\bar{u}^2dz-\frac{4}{r}\left(\int_{B_r(z_0)}\bar{u}\nabla\bar{u}\cdot(z-z_0)dz\right)^2\geq0,
\end{align*}
thanks to the Cauchy inequality. Therefore, we get
\begin{equation*}
\frac{d}{dr}N_{\bar{u}}(z_0,r)\geq\frac{1}{H^2_{\bar{u}}(z_0,r)}\left(\frac{4}{r}J_4H_{\bar{u}}(z_0,r)-\frac{1}{r}I^2_{\bar{u}}(z_0,r)\right)\geq0.
\end{equation*}
This finishes the proof.
\end{proof}

Since $\partial\Omega$ is compact and $C^{1,\alpha}$,  there exists $\delta>0$ depending only on $n$ and $\Omega$ such that the map $(x,s)\rightarrow x+s\nu(x)$ is one-to-one from $\partial\Omega\times(-\delta,0)$ onto $\delta-$neighborhood of $\partial\Omega$, where $\nu(x)$ is the outer unit normal vector of $x$. This means that
\begin{equation}
\{x+s\nu(x)\ |\ x\in\partial\Omega,s\in(-\delta,0)\}=\{\xi\in\Omega\ |\ dist(\xi,\partial\Omega)<\delta\}.
\end{equation}

Next, we should discuss under what conditions the ball $B_r(x_0)\cap\Omega$ will be star-shaped with respect to $x_0$ in $\Omega$.
\begin{lemma}\label{when is star shaped}
There exist positive constants $r_0<\delta/10$ and $C_0$ depending only on $n$ and $\Omega$, such that for any $r\in(0,r_0)$, if $dist(x_0,\partial\Omega)\geq C_0r^{1+\alpha}$, then $B_r(x_0)\cap\Omega$ is star-shaped with respect to $x_0$.
\end{lemma}

\begin{proof}
Let $a=dist(x_0,\partial\Omega)\geq C_0r^{1+\alpha}$. Firstly we choose a suitable coordinate system such that $x_0=0$, and the domain $B_r(0)\cap\Omega$ can be expressed as
\begin{equation*}
B_r(0)\cap\Omega=\{x=(x',x_n)\in B_r(0)\ |\ x_n>\gamma(x')\},
\end{equation*}
where $x'\in\mathbb{R}^{n-1}$. Moreover, it can also be supposed that $dist(0,(0,-a))=dist(0,\partial\Omega)=a$, and the outer unit normal vector at $(0,-a)$ is $(0,-1)$. So $\gamma(0)=-a$, and $\nabla\gamma(0)=0$. Since $\partial\Omega$ is $C^{1,\alpha}$ continuous, $\gamma$ is $C^{1,\alpha}$ continuous. Thus $|\nabla\gamma(x')|\leq Cr^{\alpha}$ when $|x'|\leq r$. Then for any point $x\in\partial\Omega\cap B_r(0)$, it holds
\begin{align*}
\nu(x)\cdot(x-0)&=(\nu(x)-(0,-1)+(0,-1))\cdot(x',\gamma(x'))\nonumber
\\&=(0,-1)\cdot(x',\gamma(0))+(0,-1)\cdot\left(x',\gamma(x')-\gamma(0)\right)+\left(\nu(x)-(0,-1)\right)\cdot\left(x',\gamma(x')\right)\nonumber
\\&\geq(0,-1)\cdot(x',-a)-|\gamma(x')-\gamma(0)|-|\nu(x)-(0,-1)|r\nonumber
\\&\geq a-Cr^{1+\alpha}\geq C_0r^{1+\alpha}-Cr^{1+\alpha}>0,
\end{align*}
provided that $C_0>0$ is large  depending only on $n$ and $\Omega$. So the result holds for any $r<r_0$ if $r_0\in(0,\delta/10)$ small enough.
\end{proof}

Lemma \ref{monotonicity} and Lemma \ref{when is star shaped} imply the monotonicity of frequency.
\begin{lemma}\label{upper bound of N}
For $r_1<r_2\leq r_0\leq 1$, $z_0=(x_0,0)$ with $x_0\in \Omega$, if $dist(z_0,\partial\widetilde{\Omega})\geq C_0r_0^{1+\alpha}$, there holds
\begin{equation*}
N_{\bar{u}}(z_0,r_1)\leq N_{\bar{u}}(z_0,r_2).
\end{equation*}
\end{lemma}

Based on the monotonicity  of frequency, we can establish  the doubling inequalities in the following Lemma.
\begin{lemma}\label{doubling condition}
For $r_1<r_2\leq r_0\leq1$, $z_0=(x_0,0)$ with $x_0\in \Omega$, if $dist(z_0,\partial\widetilde{\Omega})\geq C_0r_0^{1+\alpha}$, there holds
\begin{align}\label{doubling inequality}
\int_{B_{r_2}(z_0)\cap\widetilde{\Omega}}\bar{u}^2dz\leq \left(\frac{r_2}{r_1}\right)^{N_{\bar{u}}(z_0,r_2)+n+1}\int_{B_{r_1}(z_0)\cap\widetilde{\Omega}}\bar{u}^2dz,
\end{align}
\begin{align}\label{doubling inequality2}
\int_{B_{r_2}(z_0)\cap\widetilde{\Omega}}\bar{u}^2dz\geq \left(\frac{r_2}{r_1}\right)^{N_{\bar{u}}(z_0,r_1)+n+1}\int_{B_{r_1}(z_0)\cap\widetilde{\Omega}}\bar{u}^2dz.
\end{align}
\end{lemma}

\begin{proof}
From the inequality $(\ref{derivative of H})$, it holds
\begin{equation}
\frac{H'_{\bar{u}}(z_0,r)}{H_{\bar{u}}(z_0,r)}=\frac{n+1}{r}+\frac{N_{\bar{u}}(z_0,r)}{r}.
\end{equation}
So from Lemma $\ref{upper bound of N}$, for $0<r_1<r_2<r_0$,
\begin{align*}
\log\frac{H(z_0,r_2)}{H(z_0,r_1)}&=
\int_{r_1}^{r_2}\frac{H'(z_0,r)}{H(z_0,r)}dr=\int_{r_1}^{r_2}\left(\frac{n+1}{r}+\frac{N(z_0,r)}{r}\right)dr\\[2mm]
&\leq\int_{r_1}^{r_2}\frac{n+1+N(z_0,r_2)}{r}dr\\[2mm]
&\leq (n+1+N(z_0,r_2))\log\frac{r_2}{r_1},
\end{align*}
which yields
\begin{equation*}
\frac{H(z_0,r_2)}{H(z_0,r_1)}\leq\left(\frac{r_2}{r_1}\right)^{N(z_0,r_2)+n+1}.
\end{equation*}
This is (\ref{doubling inequality}).

On the other hand, from Lemma $\ref{upper bound of N}$ again, one has
\begin{align*}
\log\frac{H(z_0,r_2)}{H(z_0,r_1)}&=\int_{r_1}^{r_2}\left(\frac{n+1}{r}+\frac{N(z_0,r)}{r}\right)dr
\\[2mm]
&\geq\left(N(z_0,r_1)+n+1\right)\log\frac{r_2}{r_1},
\end{align*}
which yields (\ref{doubling inequality2}).

\end{proof}

The ``changing center'' property is as follows.
\begin{lemma}\label{changing center}
Let $r_0$ be the constant  as in  Lemma \ref{when is star shaped}. If $r\leq r_0$,  and $z_0=(x_0,0)$ with $dist(x_0,\partial\Omega)\geq C_0r^{1+\alpha}$, then for any $z_1\in B_{r/4}(z_0)$ with $z_1=(x_1,0)$ such that $dist(x_1,\partial\Omega)\geq C_0(r-a)^{1+\alpha}$
and $\rho\leq \frac{r}{2}$
it holds
\begin{equation}\label{re}
N(z_1,\rho)\leq \left(1+C\frac{a}{r}\right)N(z_0,r)+C\frac{a}{r},
\end{equation}
where $a=dist(z_0,z_1)$, and $C$ is a positive constant depending only on $n$.
\end{lemma}

\begin{proof}
Since $dist(x_0,\partial\Omega)\geq C_0r^2$ and $dist(x_1,\partial\Omega)\geq C_0(r-a)^{1+\alpha}$, Lemma $\ref{when is star shaped}$ yields that $B_r(z_0)\cap\widetilde{\Omega}$ is star-shaped with respect to $z_0$, and
$B_{r-a}(z_1)\cap\widetilde{\Omega}$ is star-shaped with respect to $z_1$.
 Then from Lemma $\ref{doubling condition}$ and the fact that $B_{r-a}(z_1)\subseteq B_r(z_0),\ B_{\frac{r}{2}-a}(z_0)\subseteq B_{r/2}(z_1)$, there holds
\begin{align}\label{estimate of r/2}
N(z_1,r/2)+n+1&\leq \log_2\frac{\int_{B_{r-a}(z_1)\cap\widetilde{\Omega}}\bar{u}^2dz}{\int_{B_{r/2}(z_1)\cap\widetilde{\Omega}}\bar{u}^2dz}
\left(\log_2\frac{r-a}{r/2}\right)^{-1}\nonumber
\\&\leq\left(1+\log_2\left(1-\frac{a}{r}\right)\right)^{-1}\log_2\frac{\int_{B_{r}(z_0)\cap\widetilde{\Omega}}\bar{u}^2dz}{\int_{B_{\frac{r}{2}-a}(z_0)\cap\widetilde{\Omega}}\bar{u}^2dz}\nonumber
\\&\leq \left(1+\log_2\left(1-\frac{a}{r}\right)\right)^{-1}(1+\log_2(1+\frac{2a}{r-2a}))(N(z_0,r)+n+1)\nonumber
\\&\leq\left(1+C\frac{a}{r}\right)(N(z_0,r)+n+1),
\end{align}
where $C$ is a positive constant depending only on $n$.
Then (\ref{re}) comes from $(\ref{estimate of r/2})$ and Lemma $\ref{upper bound of N}$.
\end{proof}


\section{Upper bound of the vanishing order}

In this section, we first introduce the doubling index and  establish the relationship between the doubling index and the frequency function, then show almost monotonicity and the upper bounds of the doubling index.
 Finally, we establish that quantitative relationship between the vanishing order and the Sobolev norm of the potential function.

\begin{definition}
We call the quantity
\begin{equation}\label{M}
M_{\bar{u}}(z_0,r)=\log_2\frac{\|\bar{u}\|^2_{L^{\infty}(B_r(z_0)\cap\widetilde{\Omega})}}{\|\bar{u}\|^2_{L^{\infty}(B_{r/2}(z_0)\cap\widetilde{\Omega})}},
\end{equation}
the doubling index of $\bar{u}$ centered at $z_0$ with radius $r$.
\end{definition}

The relationship between $N_{\bar{u}}(z_0,r)$ and $M_{\bar{u}}(z_0,r)$ is as follows.

\begin{lemma}\label{relationship}
Let $z_0=(x_0,0)$, $x_0\in\Omega$, and
assume that $B_{2r}(z_0)\cap\widetilde{\Omega}$ is star-shaped with respect to $z_0$. Then for any $\eta\in(0,1)$, it holds
\begin{equation}
M_{\bar{u}}(z_0,r)\leq \left(1+\log_2(1+\eta)\right)\cdot N_{\bar{u}}(z_0,(1+\eta)r)+C_1(1-\log_2\eta),
\end{equation}
and
\begin{equation}
M_{\bar{u}}(z_0,r)\geq \left(1-\log_2(1+\eta)\right)\cdot N_{\bar{u}}\left(z_0,\frac{1+\eta}{2}r\right)-C_2(1-\log_2\eta),
\end{equation}
where $C_1$ and $C_2$ are positive constants depending only on $n$.
\end{lemma}

To prove Lemma \ref{relationship}, we require the following improved $L^{\infty}-L^2$ estimates.
 The core of this improvement lies in the condition $\bar{V}\leq 0$, which yields the estimate  in \eqref{de} with a constant
$C$ independent of $\bar{V}$. This independence is instrumental in deriving the optimal estimates for the doubling index and measure of the nodal sets.

\begin{lemma}\label{degiorgi}
Let $\bar{u}$ be a solution of (\ref{add one dim}),  for $z_0=(x_0,0)$ with $x_0\in\Omega$, then for any $\theta\in(0,1)$, there holds
\begin{equation}\label{de}
\sup\limits_{B_{\theta r}(z_0)\cap\widetilde{\Omega}}|\bar{u}|\leq C((1-\theta)r)^{-(n+1)/2}\|\bar{u}\|_{L^2(B_r(z_0)\cap\widetilde{\Omega})},
\end{equation}
where $C$ is a positive constant depending only on $n$ and $\Omega$.
\end{lemma}

\begin{proof}
Let $\bar{v}=(\bar{u}-k)^+$ for $k\geq0$ and $\xi\in C^1_0\left(B_r(z_0)\right)$. Let $\phi=\bar{v}\xi^2$ be the test function. Since $\nabla\bar{v}=\nabla\bar{u}$ in $\{\bar{u}>k\}$, and $\bar{v}=0,\ \nabla\bar{v}=0$ in $\{\bar{u}\leq k\}$. Hence by H\"older inequality, we have
\begin{align*}
\int_{B_r(z_0)\cap\widetilde{\Omega}}\nabla\bar{u}\cdot\nabla\phi dz&=\int_{B_r(z_0)\cap\widetilde{\Omega}}\nabla\bar{u}\cdot\nabla\bar{v}\xi^2dz+\int_{B_r(z_0)\cap\widetilde{\Omega}}2\xi\bar{v}\nabla\bar{v}\cdot\nabla\xi dz
\\&\geq\int_{B_r(z_0)\cap\widetilde{\Omega}}|\nabla\bar{v}|^2\xi^2dz-2\int_{B_r(z_0)\cap\widetilde{\Omega}}|\nabla\bar{v}||\nabla\xi|\bar{v}\xi dz
\\&\geq\frac{1}{2}\int_{B_r(z_0)\cap\widetilde{\Omega}}|\nabla\bar{v}|^2\xi^2dz-2\int_{B_r(z_0)\cap\widetilde{\Omega}}\bar{v}^2|\nabla\xi|^2dz.
\end{align*}
Hence from the equation $\Delta\bar{u}+\bar{V}\bar{u}=0$ and the fact that $\bar{V}\leq 0$, it holds
\begin{align*}
\int_{B_r(z_0)\cap\widetilde{\Omega}}|\nabla\bar{v}|^2\xi^2dz&\leq\int_{B_r(z_0)\cap\widetilde{\Omega}}\nabla\bar{u}\cdot\nabla\phi dz+2\int_{B_r(z_0)\cap\widetilde{\Omega}}\bar{v}^2|\nabla\xi|^2dz
\\&=\int_{B_r(z_0)\cap\widetilde{\Omega}}\bar{V}\bar{u}\phi dz+2\int_{B_r(z_0)\cap\widetilde{\Omega}}\bar{v}^2|\nabla\xi|^2dz
\\&=\int_{B_r(z_0)\cap\widetilde{\Omega}}\bar{V}(\bar{v}+k)\bar{v}\xi^2dz+2\int_{B_r(z_0)\cap\widetilde{\Omega}}\bar{v}^2|\nabla\xi|^2dz
\\&\leq2\int_{B_r(z_0)\cap\widetilde{\Omega}}\bar{v}^2|\nabla\xi|^2dz.
\end{align*}
Recall the Sobolev inequality, $\bar{v}\xi\in W^{1,2}_0(B_r(z_0)\cap\widetilde{\Omega})$, and
\begin{equation*}
\left(\int_{B_r(z_0)\cap\widetilde{\Omega}}(\bar{v}\xi)^{2^*}dz\right)^{\frac{2}{2^*}}\leq C\int_{B_r(z_0)\cap\widetilde{\Omega}}|\nabla(\bar{v}\xi)|^2dz,
\end{equation*}
where $2^*=\frac{2(n+1)}{(n+1)-2}=\frac{2n+2}{n-1}$. Since
\begin{equation*}
\int_{B_r(z_0)\cap\widetilde{\Omega}}|\nabla(\bar{v}\xi)|^2dz\leq 2\left(\int_{B_r(z_0)\cap\widetilde{\Omega}}v^2|\nabla\xi|^2dz+\int_{B_r(z_0)\cap\widetilde{\Omega}}|\nabla\bar{v}|^2\xi^2dz\right)\leq C\int_{B_r(z_0)\cap\widetilde{\Omega}}\bar{v}^2|\nabla\xi|^2dz,
\end{equation*}
we have
\begin{equation*}
\left(\int_{B_r(z_0)\cap\widetilde{\Omega}}(\bar{v}\xi)^{2^*}dz\right)^{\frac{2}{2^*}}\leq C\int_{B_r(z_0)\cap\widetilde{\Omega}}\bar{v}^2|\nabla\xi|^2dz.
\end{equation*}
Using the H\"older inequality again, we obtain
\begin{align}\label{2}
\int_{B_r(z_0)\cap\widetilde{\Omega}}|\bar{v}\xi|^2dz&\leq C\left(\int_{B_r(z_0)\cap\widetilde{\Omega}}(\bar{v}\xi)^{2^*}dz\right)^{\frac{2}{2^*}}\left|\left\{x\in B_r(z_0)\cap\widetilde{\Omega}\ |\ \bar{v}\xi\neq0\right\}\right|^{\frac{2}{n+1}}
\nonumber\\[2mm]
&\leq C\int_{B_r(z_0)\cap\widetilde{\Omega}}\bar{v}^2|\nabla\xi|^2dz\left|\left\{x\in B_r(z_0)\cap\widetilde{\Omega}\ |\ \bar{v}\xi\neq0\right\}\right|^{\frac{2}{n+1}}.
\end{align}
 For any fixed $ 0<\theta<1$,
 choose $\xi\in C_{0}^{1}(B_r(z_0))$ such that in $\xi\equiv1$ $B_{\theta r}(z_0)$ and  $0\leq\xi\leq 1$ and $|\nabla\xi|\leq\frac{2}{(1-\theta)r}$ in $B_r(z_0)$. Let
\begin{equation*}
A(k,r)=\left\{z\in B_r(z_0)\cap\widetilde{\Omega}\ |\ \bar{u}(z)>k\right\},
\end{equation*}
then (\ref{2}) yields
\begin{equation}\label{basic iteration}
\int_{A(k,\theta r)}(\bar{u}-k)^2\leq \frac{C}{(1-\theta)^2r^2}\left|A(k,r)\right|^{\frac{2}{n+1}}\int_{A(k,r)}(\bar{u}-k)^2dz.
\end{equation}
Then, according to the De
Giorgi's iterative process (see e.g. Theorem 4.1 in \cite{hanlin}), we can arrive at the desired
estimate (\ref{de}).
\end{proof}


\textbf{Proof of Lemma $\ref{relationship}$:}
From Lemma $\ref{degiorgi}$, there holds
\begin{equation}\label{3}
\sup\limits_{B_r(z_0)\cap\widetilde{\Omega}}|\bar{u}|\leq C(\eta r)^{-\frac{n+1}{2}}\|\bar{u}\|_{L^2(B_{(1+\eta)r}(z_0)\cap\widetilde{\Omega})}.
\end{equation}
 On the other hand, it is obvious that
\begin{equation}\label{4}
\|\bar{u}\|_{L^2(B_r(z_0)\cap\widetilde{\Omega})}\leq Cr^{\frac{n+1}{2}}\sup\limits_{B_r(z_0)\cap\widetilde{\Omega}}|\bar{u}|.
\end{equation}
Applying the estimates (\ref{3}), (\ref{4}), and using (\ref{doubling inequality}), one has
\begin{align*}
M_{\bar{u}}(z_0,r)&=\log_2\frac{\|\bar{u}\|^2_{L^{\infty}(B_r(z_0)\cap\widetilde{\Omega})}}{\|\bar{u}\|^2_{L^{\infty}(B_{r/2}(z_0)\cap\widetilde{\Omega})}}\nonumber
\\&\leq\log_2\frac{C(\eta r)^{-(n+1)}\|\bar{u}\|^2_{L^2(B_{(1+\eta)r}(z_0)\cap\widetilde{\Omega})}}{C^{-1}(r/2)^{-(n+1)}\|\bar{u}\|^2_{L^2(B_{r/2}(z_0)\cap\widetilde{\Omega})}}\nonumber
\\&\leq\log_2\left(C\eta^{-(n+1)}\frac{\|\bar{u}\|^2_{L^2(B_{(1+\eta)r}(z_0)\cap\widetilde{\Omega})}}{\|\bar{u}\|^2_{L^2(B_{r/2}(z_0)\cap\widetilde{\Omega})}}\right)\nonumber
\\&\leq C-C\log_2\eta+\log_2\left(\frac{(1+\eta)r}{r/2}\right)^{N_{\bar{u}}(z_0,(1+\eta)r)+n+1}\nonumber
\\&\leq(1+\log_2(1+\eta))\cdot N_{\bar{u}}\left(z_0,(1+\eta)r\right)+C(1-\log_2\eta).
\end{align*}

On the other hand, by using (\ref{3}), (\ref{4}) again, and (\ref{doubling inequality2}), we get
\begin{align*}
M_{\bar{u}}(z_0,r)&=\log_2\frac{\|\bar{u}\|^2_{L^{\infty}(B_r(z_0)\cap\widetilde{\Omega})}}{\|\bar{u}\|^2_{L^{\infty}(B_{r/2}(z_0)\cap\widetilde{\Omega})}}\nonumber
\\&\geq\log_2\frac{C^{-1} r^{-(n+1)}\|\bar{u}\|^2_{L^2(B_{r}(z_0)\cap\widetilde{\Omega})}}
{C(\eta r/2)^{-(n+1)}\|\bar{u}\|^2_{L^2(B_{(1+\eta)r/2}(z_0)\cap\widetilde{\Omega})}}\nonumber
\\&\geq\log_2\left(C^{-1}\eta^{n+1}\frac{\|\bar{u}\|^2_{L^2(B_{r}(z_0)\cap\widetilde{\Omega})}}{\|\bar{u}\|^2_{L^2(B_{(1+\eta)r/2}(z_0)\cap\widetilde{\Omega})}}\right)\nonumber
\\&\geq C\log_2\eta-C+\log_2\left(\frac{r}{(1+\eta)r/2}\right)^{N_{\bar{u}}(z_0,(1+\eta)r/2)+n+1}\nonumber
\\&\geq (1-\log_2(1+\eta))\cdot N_{\bar{u}}\left(z_0,\frac{(1+\eta)r}{2}\right)+C(1-\log_2(1+\eta))+C(\log_2\eta-1)\\
&\geq (1-\log_2(1+\eta))\cdot N_{\bar{u}}\left(z_0,\frac{(1+\eta)r}{2}\right)-C(1-\log_2\eta).
\end{align*}
This completes the proof of Lemma \ref{relationship}.
\qed

 The doubling index is almost
monotonic in the following sense.
\begin{lemma}\label{doubling index on boundary}
For any point $z_0=(x_0,0)$ with $x_0\in \Omega$ and $r<\min\{r_0, \frac{1}{8C_0}\}$,  where $r_0$, $C_0$ are constants as in  Lemma \ref{when is star shaped},
 there holds
\begin{equation}
M_{\bar{u}}(z_0,r)\leq CM_{\bar{u}}(z_0,r_0)+C,
\end{equation}
where $C$ is a positive constant depending only on $n$ and $\Omega$.
\end{lemma}

\begin{proof}
Without loss of generality, let $z_0=(0,0)$,  we claim that
\begin{equation}\label{claim before induction}
\log_2\frac{H_{\bar{u}}(0,r)}{H_{\bar{u}}(0,r/2)}\leq (1+Cr)\log_2\frac{H_{\bar{u}}(0,2r)}{H_{\bar{u}}(0,r)}.
\end{equation}

 If $dist(z_0,\partial\widetilde{\Omega})\geq C_0(2r)^{1+\alpha}$, by using Lemma \ref{doubling condition} directly, one has
\begin{equation*}
\log_2\frac{H_{\bar{u}}(z_0,r)}{H_{\bar{u}}(z_0,r/2)}\leq \log_2\frac{H_{\bar{u}}(z_0,2r)}{H_{\bar{u}}(z_0,r)},
\end{equation*}
which implies (\ref{claim before induction}).

 If $dist(z_0,\partial\widetilde{\Omega})=dist(z_0,z^*)<C_0(2r)^{1+\alpha}$, and $\nu(z^*)=(0,-1,0)\in\mathbb{R}^{n-1}\times\mathbb{R}\times\mathbb{R}$.
Let $a=4C_0r^{1+\alpha}<\frac{r}{2}$ and $z_1=(0,a,0)$, so $dist(z_1,\partial\widetilde{\Omega})\geq a=C_0(2r)^{1+\alpha}$, then from Lemma $\ref{when is star shaped}$, $B_{2r}(z_1)\cap\widetilde{\Omega}$ is star-shaped with respect  to $z_1$. Note that $B_{\frac{r}{2}-a}(z_1)\subseteq B_{r/2}(0)$, $B_{r}(0)\subseteq B_{r+a}(z_1)$, and $B_{2r-a}(z_1)\subseteq B_{2r}(0)$, there holds
 \begin{equation*}
\log_2\frac{H_{\bar{u}}(0,r)}{H_{\bar{u}}(0,r/2)}\leq \log_2\frac{H_{\bar{u}}(z_1,r+a)}{H_{\bar{u}}(z_1,\frac{r}{2}-a)},
\end{equation*}
and
\begin{equation*}
\log_2\frac{H_{\bar{u}}(0,2r)}{H_{\bar{u}}(0,r)}\geq \log_2\frac{H_{\bar{u}}(z_1,2r-a)}{H_{\bar{u}}(z_1,r+a)}.
\end{equation*}
Then from Lemma $\ref{doubling condition}$, and the fact that $B_{2r-a}(z_1)\cap\widetilde{\Omega}$ is star-shaped, one has
\begin{equation*}
\log_2\frac{H_{\bar{u}}(z_1,r+a)}{H_{\bar{u}}(z_1,\frac{r}{2}-a)}\leq \left(1+\log_2\frac{r+a}{r-2a}\right)(N_{\bar{u}}(z_1,r+a)+n+1),
\end{equation*}
and
\begin{equation*}
\log_2\frac{H_{\bar{u}}(z_1,2r-a)}{H_{\bar{u}}(z_1,r+a)}\geq \left(1+\log_2\frac{2r-a}{2r+2a}\right)(N_{\bar{u}}(z_1,r+a)+n+1).
\end{equation*}
Therefore, we obtain
\begin{align*}
\log_2\frac{H_{\bar{u}}(0,r)}{H_{\bar{u}}(0,r/2)}&\leq\left(1+\log_2\frac{r+a}{r-2a}\right)\cdot\left(1+\log_2\frac{2r-a}{2r+2a}\right)^{-1}\log_2\frac{H_{\bar{u}}(0,2r)}{H_{\bar{u}}(0,r)}\nonumber
\\[2mm]&\leq\left(1+C\frac{a}{r}\right)\log_2\frac{H_{\bar{u}}(0,2r)}{H_{\bar{u}}(0,r)}\nonumber
\\[2mm]&\leq(1+C r^{\alpha})\log_2\frac{H_{\bar{u}}(0,2r)}{H_{\bar{u}}(0,r)},
\end{align*}
where we have used $a=4C_0r^{1+\alpha}$.
Thus the claim $(\ref{claim before induction})$ is proved.

Let $k$ be the integer such that $\frac{r_0}{2}<2^kr\leq r_0$,  by using $(\ref{claim before induction})$ for $k$ times, one gets
\begin{align}
\log_2\frac{H_{\bar{u}}(0,r)}{H_{\bar{u}}(0,r/2)}&\leq \Pi_{l=0}^{k-1}(1+C2^lr^{\alpha})\log_2\frac{H_{\bar{u}}(0,2^{k}r^{\alpha})}{H_{\bar{u}}(0,2^{k-1}r)}\nonumber
\\[2mm]
&\leq \Pi_{l=0}^{k-1}(1+C2^lr^{\alpha})\log_2\frac{H_{\bar{u}}(0,r_0)}{H_{\bar{u}}(0,r_0/4)}\nonumber
\\[2mm]&\leq \Pi_{l=0}^{k-1}e^{C2^{l-k}r_0^{\alpha}}\left(\log_2\frac{H_{\bar{u}}(0,r_0)}{H_{\bar{u}}(0,r_0/2)}
+\log_2\frac{H_{\bar{u}}(0,r_0/2)}{H_{\bar{u}}(0,r_0/4)}\right)\nonumber
\\[2mm]
&\leq C\log_2\frac{H_{\bar{u}}(0,r_0)}{H_{\bar{u}}(0,r_0/2)}.
\end{align}
Then from Lemma $\ref{degiorgi}$,
\begin{align*}
M_{\bar{u}}(0,r)&= \log_2\frac{\|\bar{u}\|^2_{L^{\infty}(B_r(0)\cap\widetilde{\Omega})}}{\|\bar{u}\|^2_{L^{\infty}(B_{r/2}(0)\cap\widetilde{\Omega})}}\leq \log_2\frac{H_{\bar{u}}(0,2r)}{H_{\bar{u}}(0,r/2)}\nonumber
\\[2mm]&\leq\log_2\frac{H_{\bar{u}}(0,2r)}{H_{\bar{u}}(0,r)}+\log_2\frac{H_{\bar{u}}(0,r)}{H_{\bar{u}}(0,r/2)}\nonumber
\\[2mm]&\leq2e^{Cr^{\alpha}}\log_2\frac{H_{\bar{u}}(0,r_0)}{H_{\bar{u}}(0,r_0/2)}\nonumber
\\[2mm]&\leq2e^{Cr^{\alpha}}(M_{\bar{u}}(0,r_0)+M_{\bar{u}}(0,r_0/2))+C.
\end{align*}
By a similar argument, it is easy to check that
\begin{equation*}
M_{\bar{u}}(0,r_0/2)\leq CM_{\bar{u}}(0,r_0)+C.
\end{equation*}
This finishes the proof.
\end{proof}

The upper bound estimate of $M_{\bar{u}}(z,r)$ is as follows.
\begin{lemma}\label{upper bound of frequency}
Let $u$ be a solution of (\ref{basic equation}).  Assume that  $V\in W^{1,\infty}(\Omega)$ and  denote $\lambda=\|V\|_{W^{1,\infty}}$, then for any $z=(x,0)$ with $x\in\Omega$
\begin{align*}
M_{\bar{u}}(z,r)\leq C\left(1+\sqrt{\lambda}\right),\quad r\leq r_0,
\end{align*}
where $C$ is a positive constant depending only on $n$ and $\Omega$.
\end{lemma}

\begin{proof}
Since $V\in W^{1,\infty}(\Omega)$, from the standard De Giorgi estimate, $u\in L^{\infty}(\Omega)$. Let $x_0$ be the maximum point of $u$ in the closure of $\Omega$. Without loss of generality, assume $|u(x_0)|=1$.
Then from the fact that $\bar{u}(x,t)=e^{\sqrt{\lambda}t}u(x)$, it is easy to see that
\begin{equation}\label{MM}
M_{\bar{u}}(z_0,r)\leq C\sqrt{\lambda}\quad \quad z_0=(x_0,0), \quad r\leq r_0.
\end{equation}
For any $z_1\in B_{r_0/4}(z_0)\cap\widetilde{\Omega}$ such that $z_1=(x_1,0)$, noting that $B_{r_0/4}(z_0)\subset B_{r_0/2}(z_1)$, using the definition of $M_{\bar{u}}(z_0,r_0)$ and  (\ref{MM}), one has
\begin{align*}
\|\bar{u}\|^2_{L^{\infty}\left(B_{r_0/2}(z_1)\cap\widetilde{\Omega}\right)} &\geq \|\bar{u}\|_{L^{\infty}\left(B_{r_0/4}(z_0)\cap\widetilde{\Omega}\right)}^2\nonumber\\
&\geq 2^{-M_{\bar{u}}(z_0,r_0/2)}\|\bar{u}\|_{L^{\infty}\left(B_{r_0/2}(z_0)\cap\widetilde{\Omega}\right)}^2\nonumber\\[2mm]
&\geq 2^{-M_{\bar{u}}(z_0,r_0/2)}2^{-M_{\bar{u}}(z_0,r_0)}\|\bar{u}\|_{L^{\infty}\left(B_{r_0}(z_0)\cap\widetilde{\Omega}\right)}^2\nonumber\\[2mm]
&\geq2^{-C\sqrt{\lambda}}.
\end{align*}
Then
\begin{equation}
M_{\bar{u}}(z_1,r_0)=\log_2\frac{\|\bar{u}\|^2_{L^{\infty}(B_{r_0}(z_1)\cap\widetilde{\Omega})}}{\|\bar{u}\|^2_{L^{\infty}(B_{r_0/2}(z_1)\cap\widetilde{\Omega})}}\leq \log_2\frac{e^{C\sqrt{\lambda}}}{2^{-C\sqrt{\lambda}}}\leq C\sqrt{\lambda}.
\end{equation}
Furthermore, using Lemma \ref{doubling index on boundary}, one has
\begin{equation}
M_{\bar{u}}(z_1,r)\leq C\sqrt{\lambda}+C, \quad  r\leq r_0.
\end{equation}
Repeat this argument for $k$ times, where $k$ is a positive integer depending only on $diam(\Omega)$, one has for any $z\in\widetilde{\Omega}$ with $z=(x,0)$,
\begin{equation*}
M_{\bar{u}}(z,r)\leq C\sqrt{\lambda}+C,\quad r\leq r_0.
\end{equation*}
This finishes the proof.
\end{proof}

At the end of this section, as a by-product, we will show   a result about the quantitative unique continuation
properties of solutions to (\ref{basic equation}). Firstly, we briefly recall some definitions.
A function $u\in L^2$ is said to have vanishing order $k\geq 0$ at some point $x_0\in\mathbb{R}^{n}$ if
\begin{align}\label{def}
 \frac{\fint_{B_{r}(x_0)}u^{2}dx}{r^{2k}}=O(1),
\end{align}
and  vanish to infinite order at point $x_0$ if
\begin{align*}
\fint_{B_r(x_0)}u^2=o(r^{2k})
\quad \mbox{for any integer} \quad k.
\end{align*}
A differential operator $L$ is said to have the strong unique continuation property in a connected domain $\Omega$ if the only solution of $Lu=0$
which vanishes to infinite order at a point $x_0\in\Omega$ is $u\equiv0$. If the strong unique continuation property holds for $L$, that means the nontrivial solutions of $Lu=0$ do not vanish of infinite order.
It is interesting to characterize the vanishing orders  of solutions by the coefficient functions appeared in the equations $Lu=0$, which is called the quantitative unique continuation property.

There is a large amount of work on strong unique continuation  for second order elliptic operators (cf. \cite{c1939, cg1999, DLW2021, h2001,G-L1987, jk1985} and references therein). For  the quantitative unique continuation property of second order elliptic operators,  Kukavica  \cite{K1998}
showed  that the vanishing order of solutions to the Schr\"{o}dinger equation:
\begin{equation*}\label{S}
-\Delta u=V(x)u.
\end{equation*}
 is less than
$C\Big(1+\|V^{-}\|_{L^{\infty}}^{\frac{1}{2}}+\mbox{osc}V+\|\nabla V\|_{L^{\infty}}\Big)$ provided $V\in W^{1,\infty}$ by using the frequency function argument.
 However, this upper bound is not sharp compared to the case  $V(x)=\lambda$.
Bakri \cite{B2012} improved the  upper bounded of vanishing order to $C\left(1+\|V\|^{\frac{1}{2}}_{C^{1}}\right)$ and Zhu \cite{zhu2016} improved to $C\left(1+\|V\|^{\frac{1}{2}}_{W^{1,\infty}}\right)$.
On the other hand, if $V(x)\in L^{\infty}$, Bourgain and Kenig \cite{BK}
showed that the  upper bounded of vanishing order is not large than $C\|V\|^{\frac{2}{3}}_{L^{\infty}}$. Moreover, Kenig \cite{K} also pointed out that the exponent $\frac{2}{3}$ of $\|V\|_{L^{\infty}}$ is sharp for complex valued $V(x)$ thanks to Meshkov's example \cite{M}.
 In \cite{K}  Kenig  asked if the vanishing order could be improved to $\|V\|^{\frac{1}{2}}_{L^{\infty}}$, matching that of the eigenfunctions, in the real-valued setting.

 We would like to point out that the results mentioned above are interior, the following quantitative continuation property is about the boundary value problem (\ref{basic equation}) parallel to them.

\begin{theorem}\label{main result 1}
Assume $\Omega\subset\mathbb{R}^{n}$ is a bounded domain with $C^{1,\alpha}$ boundary,
and $V\in W^{1,\infty}(\Omega)$.
Let $u$ be a non-trivial solution of (\ref{basic equation}).
Then the vanishing order of $u$ in $\Omega$ is  not large
than $C\left(\|V\|^{\frac{1}{2}}_{W^{1,\infty}}+1\right)$, where $C$ is a positive constant depending on $n$ and $\Omega$.
\end{theorem}
\begin{remark}
Bakri\cite{B2012} already obtained the same  maximal vanishing order  in the local sense by using the  Carleman estimate. However we derived the result for the boundary value problem \eqref{basic equation} which needs to establish a kind of  monotonicity property for vanishing orders near the domain boundary. In order to deal with this problem, we first show the monotonicity property of the frequency function at points that have some distance away from the boundary, then introduce the doubling index and give the mutually control relationships between the doubling index and the frequency function, finally show an almost monotonicity property of the doubling index to obtain the maximal vanishing order.
\end{remark}

\begin{proof}
Noting that the vanishing order of $u(x)$ and $\bar{u}(x,t)$ is the same, we will show that the maximal vanishing order of $\bar{u}$ at any point $z=(x,0)$ with $x\in\Omega$ can be controlled by  $M_{\bar{u}}(z,r)$.

Applying Lemma \ref{degiorgi},  using the definition of $M_{\bar{u}}(z,r)$ in (\ref{M}) and  Lemma \ref{doubling index on boundary}, then for any $r< r_0$,     it holds
\begin{align*}
\fint_{B_r(z)\cap\widetilde{\Omega}}\bar{u}^2dz&\geq \|\bar{u}\|^2_{L^{\infty}(B_{r/2}(z)\cap \widetilde{\Omega})}\nonumber\\[2mm]
&= 2^{-M_{\bar{u}}(z,r)}\|\bar{u}\|^2_{L^{\infty}(B_{r}(z))}\nonumber\\[2mm]
&\geq 2^{-CM_{\bar{u}}(z,r_0)-C}\|\bar{u}\|^2_{L^{\infty}(B_{r}(z))}.
\end{align*}
Furthermore, assume that there exists $k\in \mathbb{N}$ such that $2^{k}r<r_0\leq 2^{k+1}r$, one has
\begin{align*}
\|\bar{u}\|^2_{L^{\infty}(B_{r}(z))}&\geq \left(2^{-CM_{\bar{u}}(z,r_0)-C}\right)^{k+1}\|\bar{u}\|^2_{L^{\infty}(B_{2^{k+1}r}(z))}\nonumber\\[2mm]
&\geq \left(2^{-CM_{\bar{u}}(z,r_0)-C}\right)^{\log_2\frac{r_0}{r}+1} \|\bar{u}\|^2_{L^{\infty}(B_{r_0}(z))}\nonumber\\[2mm]
&=\left(\frac{r}{2r_0}\right)^{CM_{\bar{u}}(z,r_0)+C}\|\bar{u}\|^2_{L^{\infty}(B_{r_0}(z))}.
\end{align*}
Therefore,
\begin{align*}
\fint_{B_r(z)\cap\widetilde{\Omega}}\bar{u}^2dz&\geq \left(\frac{r}{4r_0}\right)^{CM_{\bar{u}}(z,r_0)+C}\|\bar{u}\|^2_{L^{\infty}(B_{r_0}(z))},
\end{align*}
which yields that the vanishing order of $\bar{u}$ at $z$ is not large than $C\left(M_{\bar{u}}(z,r_0)+1\right)$.
Moreover, by using Lemma \ref{upper bound of frequency}, we derive that the upper bound of  the vanishing order of $\bar{u}$ is
$ C\left(\|V\|^{\frac{1}{2}}_{W^{1,\infty}}+1\right)$.
This completes the proof.
\end{proof}

\section{Measure estimate of nodal sets}
In this section,  we always assume that $V$ is analytic.  We will establish measure upper bounds of nodal sets of
  $u$  in a subdomain away from the boundary and the rest subdomain  near the boundary, respectively.

First, we consider  the nodal set of $u$ in a subdomain away  from the boundary as follows.
\begin{lemma}\label{interior nodal set}
Assume $V$ is analytic, let  $r_0$, $C_0$ be    constants as in  Lemma \ref{when is star shaped},
there exists a positive constant $C$ depending only on $n$ and $\Omega$, such that for any  $0<r<\min\{\frac{r_0}{2}, \frac{1}{8C_0}\}$, it holds
\begin{equation}
\mathcal{H}^{n-1}\left\{x\in\Omega_{r}\ |\ u(x)=0\right\}\leq \frac{C}{r}\left(1+\sqrt{\lambda}\right),
\end{equation}
where $\Omega_{r}=\{x\in\Omega\ |\ dist(x,\partial\Omega)>r\}$.
\end{lemma}

\begin{proof}
From Lemma \ref{upper bound of frequency}, for any $z=(x,0)$ with $x\in\Omega_{r}$,
\begin{equation}\label{bound of M}
M_{\bar{u}}(z,2r)\leq C\left(1+\sqrt{\lambda}\right),\quad \forall\ r\leq r_0/2.
\end{equation}
Now we fix a point $x_0\in\Omega_r$, so $dist(x_0,\partial\Omega)\geq r\geq C_0(2r)^2$.  Denote $z_0=(x_0,0)$,   from Lemma $\ref{when is star shaped}$, $B_{2r}(z_0)\cap\widetilde{\Omega}$ is star-shaped with respect to $z_0$. For any $z_p\in  B_{r/8}(z_0)$ with $z_p=(x_p,0)$, it is easy to see that
\begin{equation*}
dist(x_p,\partial\Omega)\geq \frac{7}{8}r> \frac{C_0r^{1+\alpha}}{16}.
\end{equation*}
 Then by Lemma \ref{changing center} and Lemma \ref{relationship} with $\eta=\frac{1}{4}$, for any $z_p\in  B_{r/8}(z_0)$, there holds
\begin{equation*}
N_{\bar{u}}(z_p,r/4)\leq CN_{\bar{u}}(z_0,r)+C\leq CM_{\bar{u}}(z_0,8r/5)+C\leq C\left(1+\sqrt{\lambda}\right).
\end{equation*}
Therefore, by using Lemma $\ref{doubling condition}$ twice and noting $B_{r/8}(z_0)\subset B_{r/4}(z_p)$, one gets
\begin{align*}
\fint_{B_{r/16}(z_p)}\bar{u}^2dz&\geq 2^{-C(N_{\bar{u}}(z_p,r/4)+1)}\fint_{B_{r/4}(z_p)}\bar{u}^2dz
\\&\geq 2^{-C(\sqrt{\lambda}+1)}\fint_{B_{r/8}(z_0)}\bar{u}^2dz
\\&\geq 2^{-C(\sqrt{\lambda}+1)}\fint_{B_{r}(z_0)}\bar{u}^2dz.
\end{align*}
So for any ball $B_{r/16}(z_p)$ with $z_p\in B_{r/8}(z_0)$, there exists a point $z_p^*\in B_{r/16}(z_p)$, such that
\begin{equation}\label{lower of zp}
|\bar{u}(z_p^*)|\geq 2^{-C(\sqrt{\lambda}+1)}\left(\fint_{B_{r}(z_0)}\bar{u}^2dz\right)^{\frac{1}{2}}.
\end{equation}
On the other hand, from Lemma $\ref{degiorgi}$, there holds
\begin{equation}\label{upper of zp}
\|\bar{u}\|_{L^{\infty}(B_{15r/16}(z_0))}\leq C\left(\fint_{B_{r}(z_0)}\bar{u}^2dz\right)^{\frac{1}{2}}.
\end{equation}
Let $p_j=\frac{r}{4}e_j+z_0$, $j=1,2,\cdot,n+1$ with $e_j$ the $j-$th direction of the $z-$coordinates system. Then for such $p_j$, from $\label{lower of zp}$ and $\label{upper of zp}$, there exists $z_{p_j}^*\in B_{r/16}(z_{p_j})$, such that
\begin{equation}
\begin{cases}
\bar{u}(z_{p_j}^*)\geq 2^{-C(\sqrt{\lambda}+1)}\left(\fint_{B_{r}(z_0)}\bar{u}^2dz\right)^{\frac{1}{2}},\\
\|\bar{u}\|_{L^{\infty}(B_{5r/8}(z_{p_j}^*))}\leq C\left(\fint_{B_{r}(z_0)}\bar{u}^2dz\right)^{\frac{1}{2}}.
\end{cases}
\end{equation}
Since $\bar{u}$ is analytic, for any $j=1,2,\cdots,n+1$, the function $f_j(w,t)=\bar{u}(z_{p_j}^*+tw)$ is also analytic with respect to $t$ for any $w$ in $\partial B_1$ and $t\in\left(-\frac{5r}{8},\frac{5r}{8}\right)$. So one can extend $f_j(w,t)$ to $f_j(w,t+i\tau)$ with $|t|<\frac{5r}{8}$ and $|\tau|<\tau_0r$ for some positive constant $\tau_0$ depending only on $n$. So
\begin{equation}
\begin{cases}
|f_j(w,0)|\geq 2^{-C(\sqrt{\lambda}+1)}\left(\fint_{B_{r}(z_0)}\bar{u}^2dz\right)^{\frac{1}{2}},\\
|f_j(w,t+i\tau)|\leq C\left(\fint_{B_{r}(z_0)}\bar{u}^2dz\right)^{\frac{1}{2}}.
\end{cases}
\end{equation}
Applying the theory of zero distribution of analytic functions (see Lemma $6.1$ in \cite{Donnelly}),
  there holds
\begin{equation}
Card\left\{|t|<\frac{1}{2}\ |\ \bar{u}(z_{p_j}^*+tw)=0\right\}\leq C\left(\sqrt{\lambda}+1\right),
\end{equation}
and in particular, for any $w\in \partial B_1$,
\begin{equation}
Card\left\{z_{p_j}^*+tw\in B_{r/16}(z_0)\ |\ u(z_{p_j}^*+tw)=0\right\}\leq C\left(\sqrt{\lambda}+1\right).
\end{equation}
Then from the geometric formula, there holds
\begin{equation}\label{local nodal}
\mathcal{H}^{n}\left\{z\in B_{r/16}(z_0)\ |\ \bar{u}(z)=0\right\}\leq Cr^n\left(1+\sqrt{\lambda}\right).
\end{equation}
Now let $\{z_j\}$ be the points with $z_j=(x_j,0)$ and $x_j\in\Omega_{r}$, such that the balls $B_{r/32}(z_j)$ are mutually disjoint, and $\{B_{r/16}(z_j)\}$ covers $\Omega_{r}\times(-r/32,r/32)$. The number of these  balls can be controlled by $\hat{C}=C|\Omega|/r^{n}$, where $C$ is also a positive constant depending only on $n$ and $\Omega$. Then there holds
\begin{equation*}
\mathcal{H}^{n}\left\{z\in \Omega_{r}\times(-r/32,r/32)\ |\ \bar{u}(z)=0\right\}\leq C\left(1+\sqrt{\lambda}\right).
\end{equation*}
From the fact that $\bar{u}(x,t)=u(x)e^{\sqrt{\lambda}t}$, and $e^{\sqrt{\lambda}t}\neq0$, one has
\begin{align*}
&\mathcal{H}^{n-1}\left\{x\in \Omega_{r}\ |\ u(x)=0\right\}\nonumber
\\&\leq\frac{16}{r}\mathcal{H}^{n}\left\{z\in \Omega_{r}\times(-r/32,r/32)\ |\ \bar{u}(z)=0\right\}\nonumber
\\&\leq \frac{C}{r}\left(1+\sqrt{\lambda}\right).
\end{align*}
This completes the proof.
\end{proof}

In order to estimate the nodal set near the boundary, we straighten $\partial\widetilde{\Omega}$.  Let $x^{*}\in \partial\Omega$, and denote $z^{*}=(x^{*},0).$
We first translate $x^{*}$ to the origin point, so that $z^{*}=(0,0)$.  After some necessary rotation transformation,  there exists a positive constant $r_0=r_0(n,\Omega)>0$,
such that for any $r\in(0,r_0)$, the subdomain $\Omega\cap B_r(0)$ can be expressed in the following form:
\begin{equation*}
\Omega\cap B_r(0)=\{x\in B_r(0)\ |\ x_n>\gamma(x_1,\cdots,x_{n-1})\},
\end{equation*}
where $\gamma$ is the defining  $C^{1,\alpha}$ function of  $\partial\Omega$ in $B_r(0)$ with $\gamma(0)=0$, $\nabla\gamma(0)=0$, and moreover,  for any $x=(x',x_n)\in B_{r_0}(0)\cap\partial\Omega$, there holds
\begin{equation}\label{require of gamma}
|\nabla\gamma(x')-\nabla\gamma(0)|\leq C|x'|^{\alpha}.
\end{equation}
Let $z=(x,t)=(x_1,\cdots,x_n,t)$, and  define $y=\Phi(z)$ as follows,
\begin{equation*}
\left\{
\begin{array}{lll}
y_i=\Phi_i(z):=x_i,\quad i=1,\cdots,n-1,\\
y_n=\Phi_n(z)=x_n-\gamma(x_1,\cdots,x_{n-1}),\\
y_{n+1}=t.
\end{array}
\right.
\end{equation*}
The map $y=\Phi(z)$ straightens the boundary near $z_0$. Denote $z=\Psi(y)$ with $\Psi=\Phi^{-1}$, and $D=\Phi\left(\widetilde{\Omega}\cap B^z_{r_0}(0,0)\right)$.  Since $\partial\Omega$ is $C^{1,\alpha}$ continuous, there exist   $\tau_1\in(\frac{1}{2},1)$ and $\tau_2\in(1,2)$ depending only on $n$ and $\Omega$, such that
\begin{equation}\label{ball in y and z space}
B^{y}_{r/2}(y_0)\cap D\subseteq B^{y}_{\tau_1r}(y_0)\cap D \subseteq \Phi\left(B^z_r(z_0)\right)\cap D \subseteq B^{y}_{\tau_2r}(y_0)\cap D \subseteq B^{y}_{2r}(y_0)\cap D,
\end{equation}
for any $r<r_0$ and $z_0=(x_0,0) $ with $x_0\in\partial\Omega$  and $dist(x_0,0)< r_0$, $y_0=\Phi(z_0)$.
Here  $B^y_r(y_0)$ and $B^z_r(z_0)$  are the balls in $y$-space and $z$-space,
respectively. In the following, we omit the superscript $y$ and $z$ without confusion.

Let $\widetilde{u}(y)=\bar{u}(\Psi(y))$ and $\widetilde{V}(y)=\bar{V}(\Psi(y))$,   then
$\widetilde{u}(y)$ satisfies  the following equation:
\begin{equation}\label{rewrite}
\mathcal{L}\widetilde{u}+\widetilde{V}\widetilde{u}=div(L\nabla\widetilde{u})+\widetilde{V}\widetilde{u}=0, \quad \mbox{in} \quad (B_{r_0}^y(0))_+,
\end{equation}
with the boundary condition
\begin{equation*}
 \widetilde{u}=0 \quad \mbox{on}\quad  \Gamma\equiv\{y\ |\ y_n=0\}\cap (B_{r_0}^y(0))_+,
\end{equation*}
where  $L=(L_{ij})_{(n+1)\times(n+1)}$ is the following matrix:
\begin{equation}\label{def of L}
\left\{
\begin{array}{lll}
L_{ij}=\delta_{ij},\quad i,j=1,\cdots,n-1,\\[1mm]
L_{ij}=\delta_{ij},\quad i=n+1\ or\ j=n+1,\\[1mm]
L_{in}=L_{ni}=-\gamma_{y_i},\quad i=1,\cdots,n-1,\\[1mm]
L_{nn}=1+\sum\limits_{i=1}^{n-1}\gamma_{y_i}^2.
\end{array}
\right.
\end{equation}
From (\ref{definition of lambda}), we have
\begin{equation}\label{v}
\widetilde{V}(y)=\bar{V}(\Psi(y))\leq 0, \quad \mbox{and}\quad  \|\widetilde{V}\|_{L^{\infty}}=\|\bar{V}\|_{L^{\infty}}=\|V-\lambda\|_{L^{\infty}}\leq  C\|\nabla V\|_{L^{\infty}}.
\end{equation}
Moreover, $\mathcal{L}$ is uniformly elliptic, whose elliptic constants depend only on $n$ and $\Omega$, and from $(\ref{require of gamma})$,
\begin{equation}
|L_{ij}(y)-L_{ij}(0)|\leq C|y|^{\alpha},
\end{equation}
for any $y\in (B_{r_0}^y(0))_+$.

Since $\widetilde{u}(y)=\bar{u}(\Psi(y))$, for $y_0=\Phi(z_0)$ with $y_0\in Q$, we define the doubling index centered at $y_0$ with radius $r$ below:
\begin{align}\label{frequency of two phase 1}
M_{\widetilde{u}}(y_0,r)&=\log_2\frac{\|\widetilde{u}\|^2_{L^{\infty}(\Phi(B_r(z_0))\cap D)}}{\|\widetilde{u}\|^2_{L^{\infty}(\Phi(B_{r/2}(z_0))\cap D)}}
=\log_2\frac{\|\bar{u}\|^2_{L^{\infty}(B_{r}(z_0)\cap\widetilde{\Omega})}}{\|\bar{u}\|^2_{L^{\infty}(B_{r/2}(z_0)\cap\widetilde{\Omega})}}\nonumber\\[2mm]
&=M_{\bar{u}}(z_0,r).
\end{align}
So from  Lemma $\ref{doubling index on boundary}$, for any $\rho<r\leq r_0$,
\begin{equation}\label{monotonicity of M in y space}
M_{\widetilde{u}}(y_0,\rho)\leq CM_{\widetilde{u}}(y_0,r)+C,
\end{equation}
where $C$ and $r_0$ are positive constants depending only on $n$ and $\Omega$.

Moreover, we define the doubling index of $\widetilde{u}$ of a cube $Q\subseteq D$ as follows,
\begin{equation}\label{MQ}
M_{\widetilde{u}}(Q)=\max\limits_{y_0\in Q,r\leq 10(n+1) diam(Q)}M_{\widetilde{u}}(y_0,r).
\end{equation}

In the following lemma, we establish the estimate of  nodal set of $\widetilde{u}$ in some cube $Q$ by using its doubling index $M_{\widetilde{u}}(Q)$.
\begin{lemma}\label{nodal set of cube}
Assume V is analytic, let $Q$ be a cube in $\mathbb{R}^{n+1}$ with its side length $r$, and the distance between $Q$ and $\Gamma$  is greater than or equal to $10(n+1)r$, then
\begin{equation}
\mathcal{H}^n\left\{y\in Q\ |\ \widetilde{u}(y)=0\right\}\leq C(M_{\widetilde{u}}(Q)+1)r^n,
\end{equation}
where $C$ is a positive constant depending only on $n$ and $\Omega$.
\end{lemma}

\begin{proof}
Let $r$ be the side length of $Q$, and $y_0$ be the center of $Q$. Then $Q\subseteq B_{\frac{\sqrt{n+1}}{2}r}(y_0)$, and $\Psi\left(B_{\frac{\sqrt{n+1}}{2}r}(y_0)\right)\subset B_{\tau^{-1}_1\frac{\sqrt{n+1}}{2}r}(z_0)\subset B_{\sqrt{n+1}r}(z_0)$, where $z_0=\Psi(y_0)$. Since $V$ is analytic in $\Omega$, $\bar{u}$ is analytic in the interior of $\Omega$, and $B_{\sqrt{n+1}r}(z_0)\subset \widetilde{\Omega}$, 
from $(\ref{frequency of two phase 1})$, $(\ref{MQ})$, Lemma $\ref{relationship}$, we have
\begin{align*}
\mathcal{H}^n\left\{y\in Q\ |\widetilde{u}(y)=0\right\}&\leq \mathcal{H}^n\left\{y\in B_{\frac{\sqrt{n+1}}{2}r}(y_0)\ |\ \widetilde{u}(y)=0\right\}\\
&\leq C\mathcal{H}^n\left\{z\in B_{\sqrt{n+1}r}(z_0)\ |\ \bar{u}(z)=0\right\}\\
&\leq C\left(N_{\bar{u}}(z_0,2\sqrt{n+1}r)+1\right)r^{n}\\
&\leq C\left(M_{\bar{u}}(z_0,4\sqrt{n+1}r)+1\right)r^{n}\\
&\leq C\left(M_{\widetilde{u}}(y_0,4\sqrt{n+1}r)+1\right)r^{n}\\
&\leq C(M_{\widetilde{u}}(Q)+1)r^{n}.
\end{align*}
This completes the proof.
\end{proof}

\begin{remark}Let us briefly explain the process of proving Lemma \ref{nodal set of cube}.
  Because $\widetilde{u}$ satisfies a general second-order elliptic partial differential equation,  we do not directly estimate its nodal sets of  $\widetilde{u}$ in cubes $Q$ that are  far from the boundary
$\Gamma$, where the distance between $Q$ and $\Gamma$  is greater than or equal to $10(n+1)r$.
 Instead, we estimate the nodal sets of the original function $\bar{u}$, which satisfies an equation with analytic coefficients. This allows us to employ the techniques from the analytic setting to obtain a desired estimate
in which the constant $C$ depends only on $n$ and $\Omega$ rather than the potential $V$. The relation between the doubling indices of $\widetilde{u}$ and $\bar{u}$, namely $M_{\widetilde{u}}(y_0, r)= M_{\bar{u}}(z_0, r)$, then enables us to recover the nodal set estimates for $\widetilde{u}$ in $Q$.
\end{remark}

Next, we  adopt the idea of   Theorem 1.7 in \cite{Hardt}  to estimate the measure of the nodal set of $\widetilde{u}$ in some cube $Q$ touching the boundary.

\begin{lemma}\label{nodal set of cube near the boundary}
Suppose that one of the faces of cube $Q\subseteq D$ is  on $\Gamma$, and the center of this  face  is  the origin.
Let the side length of $Q$ be $r_Q$. If $r_Q\leq\|\nabla V\|^{-\frac{1}{2}}_{L^{\infty}}\left(M_{\widetilde{u}}(Q)\right)^{-CM_{\widetilde{u}}(Q)}$ with $M_{\widetilde{u}}(Q)>1$, then
\begin{equation*}
\mathcal{H}^n\left\{y\in Q\ |\ \widetilde{u}(y)=0\right\}\leq C'M_{\widetilde{u}}(Q)r_Q^n,
\end{equation*}
where $C$ and $C'$ are positive constants depending only on $n$ and $\Omega$.
\end{lemma}
\begin{remark}
Lemma \ref{nodal set of cube near the boundary}  also holds for the case of $V=constant$. Under the case $V=constant$, the side length $r_Q$ of cube $Q$,
 no longer has a size constraint. Comparing  with the
eigenvalue problem, the potential term really bring new difficulties, we can only estimate the measure of the nodal set in the  cube $Q$ with small side length depending on  the gradient of $V$. This is  why the  term $\log \|\nabla V\|_{L^\infty}$ will appears   in the upper bound of the nodal sets (see {\it Step 6} in the proof of Theorem 1.1 at the last of this section).
\end{remark}
\begin{proof}
Let $Q'=\left\{y=(y',y_n,y_{n+1})\ |\ (y',-y_n,y_{n+1})\in Q\right\}$. We extend $\widetilde{u}$ to $Q'$ by
\begin{equation*}
\widetilde{u}(y',y_n,y_{n+1})=-\widetilde{u}(y',-y_n,y_{n+1}),\ y=(y', y_n,y_{n+1})\in Q'.
\end{equation*}
Then in $\bar{Q}\cup\bar{Q'}$, $\widetilde{u}$ satisfies the following equation:
\begin{equation*}
\mathcal{L}^{\#}\widetilde{u}+\widetilde{V}\widetilde{u}=div(L^{\#}\widetilde{u})+\widetilde{V}\widetilde{u}=0,
\end{equation*}
where $L^{\#}=(L^{\#}_{ij})_{(n+1)\times(n+1)}$ is the symmetric positive definite matrix
\begin{equation}\label{eq}
\left\{
\begin{array}{lll}
L^{\#}_{ij}=\delta_{ij},\quad i,j=1,\cdots,n-1,\\
L^{\#}_{ij}=\delta_{ij},\quad i=n+1\ or\ j=n+1,\\
L^{\#}_{in}=L^{\#}_{ni}=-\gamma_{y_i}\cdot(sgn(y_n)),\quad i=1,\cdots,n-1,\\
L^{\#}_{nn}=1+\sum\limits_{i=1}^{n-1}\gamma_{y_i}^2,
\end{array}
\right.
\end{equation}
where $sgn(y_n)$ is the sign function  of $y_n$.  We note that $L^{\#}_{ij}$ is continuous at the origin thanks to (\ref{eq}) and the properties of $\gamma$ in (\ref{require of gamma}).
From Theorem \ref{main result 1} and the definition of  $M_{\widetilde{u}}(Q)$, we see that the vanishing order of $\widetilde{u}$ at any point of $\bar{Q}\cup\bar{Q'}$ is less than or equal to $CM_{\widetilde{u}}(Q)$. This means that for any point $y\in \bar{Q}\cup\bar{Q'}$, if there exists some $d>0$ such that
\begin{equation*}
\limsup\limits_{r\rightarrow0}r^{-d}\|\widetilde{u}\|_{L^{\infty}(B_r(y))}=0,
\end{equation*}
then $d\leq CM_{\widetilde{u}}(Q)$.  Since we have required that
$r_Q\leq\|\nabla V\|^{-\frac{1}{2}}_{L^{\infty}}\left(M_{\widetilde{u}}(Q)\right)^{-CM_{\widetilde{u}}(Q)}$, the conditions of Theorem $1.7$ in \cite{Hardt} are satisfied. Then we have
\begin{equation*}
\mathcal{H}^n\left\{y\in \bar{Q}\cup\bar{Q'}\ |\ \widetilde{u}(y)=0\right\}\leq CM_{\widetilde{u}}(Q)r_Q^n,
\end{equation*}
which yields the desired result directly.
\end{proof}

Next, we shall establish a  propagation property of smallness of solutions of (\ref{rewrite}) in some cube. In fact, there are various results  about the propagation of smallness of the solutions of  linear elliptic equations of second-order (e.g. Theorem 1.7 in \cite{stability}). In the following Lemma,  we will
focus on how the side length of the cube explicitly  depends on the norms of the potential.

\begin{lemma}\label{cauchy}
Let $\mathcal{L}$ be the same operator as in $(\ref{rewrite})$, and  $Q\subseteq D$ be a given cube  with its side length $R\leq R_0\|\nabla V\|_{L^{\infty}}^{-\frac{1}{2}}$, where $R_0$ is a positive constant depending only on $n$ and $\mathcal{L}$. Let $F$ be a given face of $Q$. Let $\widetilde{u}$ be a solution of (\ref{rewrite}), and  assume  that $|\widetilde{u}|\leq 1$ in $Q$. Then there exist positive constants $C$ and $\alpha<1$, depending only on $\mathcal{L}$ and $n$, such that if $|\widetilde{u}|\leq \epsilon$ on $F$, and $|\nabla\widetilde{u}|\leq \epsilon/R$ on $F$, $\epsilon<1$, then
\begin{equation}
\sup\limits_{\frac{1}{2}Q}|\widetilde{u}|\leq C\epsilon^{\alpha}.
\end{equation}
\end{lemma}

\begin{proof}
Let $w$ be the unique solution to the following equation:
\begin{equation}\label{equ-w}
\left\{
\begin{array}{lll}
\mathcal{L}w+\widetilde{V}w=0,\quad \mbox{in}\  Q,\\[1mm]
w=1,\quad \mbox{on}\ \partial Q.
\end{array}
\right.
\end{equation}
Denoting $\bar{w}=w-1$, we have
\begin{equation}\label{bar}
\left\{
\begin{array}{lll}
\mathcal{L}\bar{w}+\widetilde{V}\bar{w}=-\widetilde{V},\quad \mbox{in}\  Q,\\[1mm]
\bar{w}=0,\quad \mbox{on}\  \partial Q,
\end{array}
\right.
\end{equation}
By standard a-priori bounds in $L^{\infty}$, see Theorem 8.16 in \cite{GT}, and noting that $\widetilde{V}\leq 0$,
we have
\begin{equation*}
\|\bar{w}\|_{L^{\infty}(Q)}\leq CR^2\|\widetilde{V}\|_{L^{\infty}}\leq CR^2\|\nabla V\|_{L^{\infty}},
\end{equation*}
where we used the fact that $\|\widetilde{V}\|_{L^{\infty}}\leq C\|\nabla V\|_{L^{\infty}}$.
So by choosing $R\leq R_0\|\nabla V\|^{-\frac{1}{2}}_{L^{\infty}}$ with $R_0>0$ small enough, there holds
\begin{equation}\label{w1}
\|\bar{w}\|_{L^{\infty}(Q)}\leq \frac{1}{2}.
\end{equation}
Furthermore, applying  a global Schauder type estimate to (\ref{bar}), see Theorem 8.33 in  \cite{GT},
one obtains,
\begin{align}\label{w2}
\|\nabla\bar{w}\|_{L^{\infty}(Q)}&\leq C\left\{R^{-1}\|\bar{w}\|_{{L^{\infty}}(Q)}+R\|\widetilde{V}(1-\bar{w})\|_{{L^{\infty}}(Q)}\right\}\nonumber\\[2mm]
&\leq C\left(R^{-1}+R\|\widetilde{V}\|_{{L^{\infty}}(Q)}\right)\nonumber\\[2mm]
&\leq C \left(R^{-1}+R_0\|\nabla V\|_{{L^{\infty}}}^{\frac{1}{2}}\right).
\end{align}
Let $\widetilde{u}=vw$.
Then $v=\widetilde{u}/w$ satisfies the following equation in $Q$:
\begin{equation}\label{equ-v}
\mathcal{L}^*v=div(L^*\nabla v)=0,
\end{equation}
where $L^*=w^2L$. Through some direct calculation and the assumption of $R\leq R_0\|\nabla V\|^{-\frac{1}{2}}_{L^{\infty}}$,  in virtue of (\ref{w1}) and (\ref{w2}), one has
\begin{equation}
|v|\leq 2\epsilon\quad \mbox{on}\ F, \quad \mbox{and}\quad  |\nabla v|\leq 6\epsilon/R \quad \mbox{on}\ F.
\end{equation}
Indeed,
\begin{align*}
|\nabla v|&\leq \left|\frac{\nabla \widetilde{u}}{w}\right|+|\widetilde{u}|\left|\frac{\nabla w}{w^2}\right|\\[2mm]
&\leq\frac{C\epsilon} {R}.
\end{align*}
Then, applying Theorem 1.7 in \cite{stability} to the case of homogeneous equation (\ref{equ-v}), we infer
\begin{equation*}
\sup\limits_{\frac{1}{2}Q}|v|\leq C\epsilon^{\alpha},
\end{equation*}
where $C$ and $\alpha$ both are positive constants depending only on $n$ and $\mathcal{L}$.
This completes the proof  thanks to  $|\widetilde{u}|=|vw|\leq \frac{3}{2}|v|$.
\end{proof}

Given an Euclidean $(n+1)$-dimensional cube $Q$.
  Let $A\geq1$ be an odd integer, we divide $Q$ into $A^{n+1}$ equal subcubes.
   By using the small data propagation,
we prove the following dividing lemma  to estimate the doubling indexes of subcubes.

\begin{lemma}\label{first dividing lemma}
Given an Euclidean $(n+1)$-dimensional cube $Q\subseteq D$, which side length is  $R$, and assume that one face of $Q$ are parallel to $\Gamma$.
  Let $A\geq1$ be an odd integer, we partition $Q$ into $A^{n+1}$ equal subcubes.
Let $q_{i,0}$ be any layer of subcubes of $Q$.
There exist $A_0$, $R_0$, and $M_0>1$ depending only on $n$ and $\mathcal{L}$, such that if $A\geq A_0$, $M_{\widetilde{u}}(Q)\geq M_0$, $R\leq R_0\|\nabla V\|_{L^{\infty}}^{-\frac{1}{2}}$, then there exists at least one cube of $q_{i,0}$ satisfying
\begin{equation*}
M_{\widetilde{u}}(q_{i,0})\leq \frac{M_{\widetilde{u}}(Q)}{2}.
\end{equation*}
\end{lemma}

\begin{proof}
Fix a layer of cubes $q_{i,0}$. Let $y^{*}$ be the center of this layer of cubes. Also let $B$ be the ball centered at $y^{*}$ with radius $R$.
Let $M=\frac{M_{\widetilde{u}}(Q)}{2}\geq \frac{M_0}{2}$. We prove this lemma by contradiction. Without loss of generality, we assume that
$\sup_{\frac{1}{8}B\cap D}|\widetilde{u}|=1.$
By the contradiction assumptions, for any subcube $q_{i,0}$, there exist $y_i\in q_{i,0}$ and $r_i<10(n+1)diam(q_{i,0})$ such that  $M_{\widetilde{u}}(y_i,r_i)>M$. Since $2q_{i,0}\subseteq B_{\sqrt{n+1}\frac{2R}{A}}(y_i)$ and $B_{R/32}(y_i)\subseteq B$, in virtue of (\ref{ball in y and z space}) and  (\ref{monotonicity of M in y space}), we have
\begin{align*}
\sup\limits_{2q_{i,0}}|\widetilde{u}|&\leq\sup\limits_{B_{\sqrt{n+1}\frac{2R}{A}}(y_i)\cap D}|\widetilde{u}|\\
&\leq \sup\limits_{\Phi\left(B_{\sqrt{n+1}\frac{4R}{A}}(z_i)\right)\cap D}|\widetilde{u}|\\
&\leq\sup\limits_{\Phi\left(B_{R/64}(z_i)\right)\cap D}|\widetilde{u}| 2^{-CM\log_2 A}\\
&\leq \sup\limits_{B_{R/32}(y_i)\cap D}|\widetilde{u}|2^{-CM\log_2 A}\\
&\leq\sup\limits_{B\cap D}|\widetilde{u}|2^{-CM\log_2 A}\\
&\leq\sup\limits_{\Phi(B_{2R}(z^*)\cap D)}|\widetilde{u}|2^{-CM\log_2 A}\\
&\leq\sup\limits_{\Phi\left(B_{R/16}(z^*)\right)\cap D}|\widetilde{u}|2^{-CM\log_2 A}\\
&\leq\sup\limits_{\frac{1}{8}B\cap D}|\widetilde{u}|2^{-CM\log_2 A}\\
&\leq 2^{-CM\log_2 A},
\end{align*}
where $y_i=\Phi(z_i)$, and $y^*=\Phi(z^*)$.
Then by the standard $W^{2,p}$ elliptic estimate (see Theorem 9.11 in \cite{GT})
to $v=\widetilde{u}/w$, where $v$ satisfies equation (\ref{equ-v}), one has
\begin{align}
\sup\limits_{q_{i,0}}|\nabla \widetilde{u}|&\leq C\left(\sup\limits_{q_{i,0}}|\nabla v|+R^{-1}\sup\limits_{q_{i,0}}|v|\right)\leq CAR^{-1}\sup\limits_{2q_{i,0}}|v|\nonumber
\\&\leq CAR^{-1}\sup\limits_{2q_{i,0}}|\widetilde{u}|\nonumber
\\&\leq R^{-1}2^{-CM\log_2 A}.
\end{align}
So  $|\widetilde{u}|\leq \epsilon $ and $|\nabla\widetilde{u}|\leq R^{-1}\epsilon$ on $B\cap\{y_n=y^{*}_n\}$, where $\epsilon=2^{-CM\log_2 A}$.

Let $q\subseteq Q$ be a cube with side $\frac{R}{16\sqrt{n+1}}$, whose faces are parallel to $\Gamma$ and centered at either $y^{*}+\left(0,\frac{3}{8}R,0\right)$ or $y^{*}-\left(0,\frac{3}{8}R,0\right)$ depending on the layer belonging to uphalf or lowerhalf of the cube $Q$, denoting by $y_0$.  Then from Lemma $\ref{cauchy}$,
\begin{equation}
\sup\limits_{q}|\widetilde{u}|
\leq 2^{-C_1M\alpha\log_2 A}.
\end{equation}
Hence from the fact that $\frac{1}{8}B\subseteq B_{\frac{R}{2}}(y_0)$ and $B_{\frac{R}{32\sqrt{n+1}}}(y_0)\subseteq q$, it holds
\begin{equation}
\frac{\sup\limits_{\Phi(B_{R}(z_0))\cap D}|\widetilde{u}|^2}{\sup\limits_{\Phi\left(B_{\frac{R}{64\sqrt{n+1}}}(z_0)\right)\cap D}|\widetilde{u}|^2}\geq\frac{\sup\limits_{B_{\frac{R}{2}}(y_0)\cap D}|\widetilde{u}|^2}{\sup\limits_{B_{\frac{R}{32\sqrt{n+1}}}(y_0)\cap D}|\widetilde{u}|^2}\geq 2^{C_1M\alpha\log_2 A}.
\end{equation}
On the other hand, from the form $(\ref{monotonicity of M in y space})$,
\begin{equation}
\frac{\sup\limits_{\Phi\left(B_{R}(z_0)\right)\cap D}|\widetilde{u}|^2}{\sup\limits_{\Phi\left(B_{\frac{R}{64\sqrt{n+1}}}(z_0)\right)\cap D}|\widetilde{u}|^2}\leq
2^{CM_{\widetilde{u}}(y_0,R/2)+C}.
\end{equation}
So
\begin{equation}
M_{\widetilde{u}}(Q)\geq M_{\widetilde{u}}(y_0,R/2)\geq C_2M\alpha\log_2 A-C.
\end{equation}
 Choosing $A\geq A_0$ and $M\geq M_0/2$ with positive constants $M_0$ and $A_0$ large enough such that $M_{\widetilde{u}}(Q)> 2M$ holds. This is a contradiction.
\end{proof}

Now we are ready to prove our main result.

\textbf{Proof of  Theorem $\ref{main result 2}$.}\
{\it Step $1$.} Let $Q=[0, R]^{n}\times[-\frac{R}{2}, \frac{R}{2}]\subseteq D$ be a given cube with $R\leq R_0\|\nabla V\|_{L^{\infty}}^{-\frac{1}{2}}$. Also assume that $M_{\widetilde{u}}(Q)$ is large enough such that $M_{\widetilde{u}}(Q)>M_0$, where $M_0>1$ is the same positive constant as in Lemma $\ref{first dividing lemma}$.
Divide $Q$ into $A^{n+1}$ equal subcubes, where $A$ is the same  integer as in Lemma $\ref{first dividing lemma}$. We use $q_{i_1j_1}$, $i_1=1,2,\cdots,\ A$ and $j_1=1,\cdots,A^n$ to denote the first  generation subcubes. Here the subscript $i_1$ means  $i_1$-th layer of subcubes from top to bottom, i.e., $q_{1j_1}$ is a subcube on the top layer, and $q_{Aj_1}$ is a subcube on the bottom layer. There are $A^n$ first  generation subcubes in each layer.  Then, according to Lemma $\ref{first dividing lemma}$, we see that  there exists at least one first generation subcube in each layer where the  doubling index of $\widetilde{u}$ is less than or equal to $\frac{M_{\widetilde{u}}(Q)}{2}$.  Let $A'=A-10(n+1)$,
 we can apply  Lemma $\ref{nodal set of cube}$  to the  first  generation subcubes $q_{i_1j_1}$ in each $i_1=1,2,\cdots, A'$ layer,
\begin{align}\label{near boundary 1}
\sum\limits_{j_1=1}^{A^n}\mathcal{H}^n\left\{y\in q_{i_1j_1} \ |\ \widetilde{u}(y)=0\right\}
&\leq C \sum\limits_{j_1=1}^{A^n} M_{\widetilde{u}}(q_{i_1j_1})\left(\frac{R}{A}\right)^{n}\nonumber\\[2mm]
&\leq C\left[\frac{M_{\widetilde{u}}(Q)}{2}R^{n}A^{-n}+(A^n-1)M_{\widetilde{u}}(Q)R^nA^{-n}\right]\nonumber\\[2mm]
&=CR^n\left(1-\frac{1}{2}A^{-n}\right)M_{\widetilde{u}}(Q)\nonumber\\[2mm]
&=:C\kappa R^nM_{\widetilde{u}}(Q),
\end{align}
where $\kappa=1-\frac{1}{2}A^{-n}$.
Since $A'<A$, we establish the upper bounded of measure of the nodal set in
the  first  generation subcubes in  the top  $A'$ layers as
\begin{equation*}
\sum\limits_{i_1=1}^{A'}\sum\limits_{j_1=1}^{A^n}\mathcal{H}^n\left\{y\in q_{i_1j_1} \ |\ \widetilde{u}(y)=0\right\}\leq CAR^n\kappa M_{\widetilde{u}}(Q).
\end{equation*}

{\it Step $2$.}
For the remaining $10(n+1)$ layers of the  first  generation subcubes, we divide each $q_{i_1j_1}$ ($i_1=A'+1,\cdots,A$,  $j_1=1\cdots,A^n$) into $A^{n+1}$ equal subcubes of the second generation,
denoting them   by $q_{i_1j_1,i_2j_2}$. Here $i_2$ also means the $i_2$-th layer  from top to bottom, just the same as $i_1$, $j_2$ is similar to $j_1$. Similarly,
for each layer, there are $A^{2n}$   second  generation subcubes, and the side length of them is $\frac{R}{A^2}$.
In order to estimate the upper bound of the measure of the nodal set  of such subcubes,
from Lemma $\ref{first dividing lemma}$,  we  see that  the number of the second  generation subcubes whose doubling index is less than or equal to $M_{\widetilde{u}}(Q)/4$ is at least $1$; the number of the second  generation subcubes whoses doubling index between $M_{\widetilde{u}}(Q)/4$ and $M_{\widetilde{u}}(Q)/2$ is $2(A^n-1)$; the number of the second  generation subcubes whose doubling index is $M_{\widetilde{u}}(Q)$ is at most $(A^n-1)^2$ in each layer.
Then, applying Lemma $\ref{nodal set of cube}$ again, for each layer,
\begin{align*}
&\sum\limits_{j_1=1}^{A^n}\sum\limits_{j_2=1}^{A^n}\mathcal{H}^n\left\{y\in q_{i_1j_1,i_2j_2} \ |\ \widetilde{u}(y)=0\right\}\nonumber
\\[2mm]
&\leq C\left\{\frac{M_{\widetilde{u}}(Q)}{4}\left(\frac{R}{A^2}\right)^{n}+2(A^{n}-1)\cdot\frac{M_{\widetilde{u}}(Q)}{2}\left(\frac{R}{A^2}\right)^{n}
+\left(A^{n}-1\right)^2\cdot M_{\widetilde{u}}(Q)\left(\frac{R}{A^2}\right)^{n}\right\}\nonumber\\[2mm]
&=CR^{n}M_{\widetilde{u}}(Q)(1-\frac{1}{2}A^{-n})^2\nonumber\\[2mm]
&=CR^{n}\kappa^2M_{\widetilde{u}}(Q).
\end{align*}
Therefore, for  second generation subcubes in all layers of $i_1=A'+1,\cdots, A-1$, $i_2=1,2,..., A$  and $i_1=A, i_2=1,2,\cdots, A'$, that is, despite the bottom $10(n+1)$ layers of the second generation subcubes,
 there holds
\begin{align}\label{fin1}
&\sum\limits_{i_1=A'+1}^{A-1}\sum\limits_{j_1=1}^{A^n}\sum\limits_{i_2=1}^{A}\sum\limits_{j_2=1}^{A^n}\mathcal{H}^n\left\{y\in q_{i_1j_1,i_2j_2} \ |\ \widetilde{u}(y)=0\right\}
+\sum\limits_{i_2=1}^{A'}\sum\limits_{j_1=1}^{A^n}\sum\limits_{j_2=1}^{A^n}\mathcal{H}^n\left\{y\in q_{Aj_1,i_2j_2} \ |\ \widetilde{u}(y)=0\right\}\nonumber
\\&\leq CAR^n\kappa^2M_{\widetilde{u}}(Q).
\end{align}
Let $k_0=[\log_2\frac{M_{\widetilde{u}}(Q)}{M_0}]+1$.
Repeating the same arguments until the $k_0$-th partition. We note that, after the $k_0$-th partition, there may exist at least one subcube whose doubling index is less than or equal to $M_0$. Then Lemma $\ref{first dividing lemma}$ cannot be used again for these subcubes.

For the  $k$ $(k\geq k_0)$-th  partitions, the doubling index of some  $k$-th generation subcubes are less than or equal to $M_0$, and the side length of the  $k$-th generation subcubes is  $RA^{-k}\leq RA^{-k_0}$. Moreover, if we choose $A=A_0$ large enough such that
\begin{equation}\label{require of A0}
A_0\geq M_0^{CM_0}.
\end{equation}
Then  the side length of the $k$-th  generation subcubes $RA^{-k}$    is small in the following sense,
\begin{align}\label{smallr}
RA^{-k}&\leq RA^{-k_0}\leq CR_0\|\nabla V\|^{-\frac{1}{2}}_{L^{\infty}}\left(\frac{M_0}{2M_{\widetilde{u}}(Q)}\right)^{\log_2A}
\nonumber\\[2mm]
&\leq CR_0\|\nabla V\|^{-\frac{1}{2}}_{L^{\infty}}\left(
\frac{1}{2}\right)^{CM_0\log_2M_0}\nonumber\\[2mm]
&\leq  C\|\nabla V\|^{-\frac{1}{2}}_{L^{\infty}}M_0^{-CM_0}.
\end{align}
For the subcubes whose doubling index is greater than $M_0$, we divide them again and continue the above process.

{\it Step $3$.}
For a  $k$-th generation subcube $\widehat{Q}$  whose doubling index is less than or equal to $M_0$,  denote the side length of $\widehat{Q}$  be $r_{\widehat{Q}}$,  we claim that
\begin{equation}\label{both}
\mathcal{H}^n\left(\{y\in \widehat{Q}\ |\ \widetilde{u}(y)=0\}\right)\leq  CAM_0r_{\widehat{Q}}^{n}.
\end{equation}
We shall prove this claim  in two cases.

{\it Case 1.}\ If  $\widehat{Q}$  is not on the bottom layer,
 we divide $\widehat{Q}$ into $A^{n+1}$ subcubes $\widehat{Q}_i$. The doubling index of every  subcubes $\widehat{Q}_i$ is less than or equal to $M_0$. Therefore,   applying Lemma $\ref{nodal set of cube}$ in every subcubes $\widehat{Q}_i$, one has
\begin{align*}
&\mathcal{H}^n\left(\{y\in \widehat{Q}\ |\ \widetilde{u}(y)=0\}\right)\leq CA^{n+1} (M_0+1)\left(\frac{r_{\widehat{Q}}}{A}\right)^{n}\nonumber\\[2mm]
&\leq CAM_0r_{\widehat{Q}}^{n}.
\end{align*}

{\it Case 2.} \ If  $\widehat{Q}$ is on the last layer, that is one face of $\widehat{Q}$ lies on $\Gamma$, we have
\begin{align}\label{case2}
&\mathcal{H}^n\left(\{y\in \widehat{Q}\ |\ \widetilde{u}(y)=0\}\right)\leq \mathcal{H}^n\left(\left\{y\in (B^{y}_{\frac{\sqrt{n+1}}{2}r_{\widehat{Q}}}(y_0))_+\ |\ \widetilde{u}(y)=0\right\}\right)\nonumber
\\&\leq C\mathcal{H}^n\left(\left\{z\in B^{z}_{\tau_1^{-1}\frac{\sqrt{n+1}}{2}r_{\widehat{Q}}}(z_0)\cap\widetilde{\Omega}\ |\ \bar{u}(z)=0\right\}\right),
\end{align}
where $y_0$ is the center of the face on $\Gamma$ of $\widehat{Q}$, $z_0=\Psi(y_0)$, $r_{\widehat{Q}}$ is the side length of $\widehat{Q}$. Then, by translating $z_0$ to the origin point and straight $\partial\Omega$  in a neighborhood of $z_0$,  we get a new function $\widetilde{u}^*(y^*)$ on $y^*-$space. Therefore,
\begin{align}\label{y1}
&\mathcal{H}^n\left(\{z\in B^{z}_{\tau_1^{-1}\frac{\sqrt{n+1}}{2}r_{\widehat{Q}}}(z_0)\cap\widetilde{\Omega}\ |\ \bar{u}=0\}\right)\nonumber
\\&\leq C\mathcal{H}^n\left(\{y^*\in (B^{y^*}_{\tau_2\tau_1^{-1}\frac{\sqrt{n+1}}{2}r_{\widehat{Q}}}(0)_+\ |\ \widetilde{u}^*=0\}\right)\nonumber
\\&\leq C\mathcal{H}^n\left(\{y^*\in q^*\ |\ \widetilde{u}^*=0\}\right),
\end{align}
where
$q^*$ is the cube in $y^{*}$-space with side length $r^*=2\tau_2\tau_1^{-1}\frac{\sqrt{n+1}}{2}r_{\widehat{Q}}$, and  the center of the face on the boundary is $0$.
Noting that $r^*\leq Cr_{\widehat{Q}}\leq CRA^{-k}$ is small in the sense of (\ref{smallr}), then we can  apply Lemma $\ref{nodal set of cube near the boundary}$  to $\widetilde{u}^*$ in $q^{*}$,
\begin{align}\label{y2}
&\mathcal{H}^n\left(\{y^*\in q^*\ |\ \widetilde{u}^*=0\}\right)\nonumber\\
&\leq CM_0\cdot\left(2\tau_2\tau_1^{-1}\frac{\sqrt{n+1}}{2}r_{\widehat{Q}}\right)^n\nonumber
\\&\leq CM_0r_{\widehat{Q}}^n.
\end{align}
Putting (\ref{y2}) back to (\ref{y1}) and (\ref{case2}) implies that
\begin{equation*}
\mathcal{H}^n\left(\{y\in \widehat{Q}\ |\ \widetilde{u}(y)=0\}\right)\leq  CM_0r_{\widehat{Q}}^n.
\end{equation*}
Therefore, in both case, we have proved the claim (\ref{both}).

 For the $k$-th  partition, we only compute measures of nodal sets for the  $k$-th generation subcubes of at most $10(n+1)A$ bottom layers. On each layer, the number of the $k$-th generation subcubes whose doubling indexes are less than or equal to $M_0$ is less than or equal to $\mathcal{C}_k^{k_0}(A^n-1)^{k-k_0}$. So on the $k-$th ($k\geq k_0$) partition, the measure of the nodal set of  the union of these $k$-th generation subcubes is not large than
\begin{equation}\label{fin2}
CA^2\mathcal{C}_k^{k_0}(A^n-1)^{k-k_0}M_0\left(\frac{R}{A^{k}}\right)^{n}.
\end{equation}

{\it Step 4.}\ Repeating the above  partition procedure infinitely many times, for the class of $k$-th subcubes whose doubling index are large than $M_0$, we can estimate the nodal set  of the union of these $k$-th generation subcubes (simply denoted by $Q_{k}$) similarly as in (\ref{fin1}), that is
\begin{equation}
\mathcal{H}^{n} \left\{y\in Q_{k}\ |\ \widetilde{u}(y)=0\right\}
\leq CAR^n\kappa^kM_{\widetilde{u}}(Q).
\end{equation}

On the other hand, for the class of $k$-th ($k\geq k_0$) subcubes whose doubling index is less than or equal to $M_0$, we can estimate the nodal set  of the union of these $k$-th generation subcubes (simply denoted by $Q'_{k}$)  as in (\ref{fin2}), that is
\begin{equation}
\mathcal{H}^{n} \left\{y\in Q'_{k}\ |\ \widetilde{u}(y)=0\right\}
\leq CA^2\mathcal{C}_k^{k_0}(A^n-1)^{k-k_0}M_0\left(\frac{R}{A^{k}}\right)^{n}.
\end{equation}
Therefore,
noting that $M_0\leq\frac{M_{\widetilde{u}}(Q)}{2^{k_0-1}}$ and  $\kappa=1-\frac{1}{2}A^{-n}$, it holds
\begin{align*}
&\mathcal{H}^n\left\{y\in Q\ |\ \widetilde{u}(y)=0\right\}\nonumber
\\&\leq \sum\limits_{k=1}^{\infty}CAR^nM_{\widetilde{u}}(Q)\kappa^k
+\sum\limits_{k=k_0}^{\infty}CA^2\mathcal{C}_k^{k_0}(A^n-1)^{k-k_0}M_0\left(\frac{R}{A^{k}}\right)^{n}
\nonumber
\\&\leq CAR^nM_{\widetilde{u}}(Q)\sum\limits_{k=1}^{\infty}\kappa^k
+CA^2R^n\frac{M_{\widetilde{u}}(Q)}{2^{k_0}}\sum\limits_{k=k_0}^{\infty}\mathcal{C}_k^{k_0}(1-A^{-n})^{k-k_0}A^{-nk_0}\nonumber
\\&\leq CAR^nM_{\widetilde{u}}(Q)\sum\limits_{k=1}^{\infty}\kappa^k+CA^2R^nM_{\widetilde{u}}(Q)\sum\limits_{k=k_0}^{\infty}\mathcal{C}_k^{k_0}\left(\frac{1}{2}A^{-n}\right)^{k_0}(1-A^{-n})^{k-k_0}\nonumber
\\&\leq CAR^nM_{\widetilde{u}}(Q)\sum\limits_{k=1}^{\infty}\kappa^k+CA^2R^nM_{\widetilde{u}}(Q)\sum\limits_{k=k_0}^{\infty}\kappa^k\nonumber
\\&\leq CAR^nM_{\widetilde{u}}(Q)\frac{\kappa}{1-\kappa}+CA^2R^nM_{\widetilde{u}}(Q)\frac{\kappa^{k_0}}{1-\kappa}\nonumber
\\&\leq CR^nA^{n+2}M_{\widetilde{u}}(Q)\nonumber
\\&\leq CR^nA^{n+2}(\sqrt{\lambda}+1),
\end{align*}
where we used $M_{\widetilde{u}}(Q)\leq C(\sqrt{\lambda}+1)$ in the last inequality.

{\it Step 5.} Using the map $\Psi$ mapping  $y$ back to $z$, one obtains that
\begin{equation*}
\mathcal{H}^n\left\{z\in\Psi(Q)\ |\ \bar{u}(z)=0\right\}\leq CR^nA^{n+2}(\sqrt{\lambda}+1).
\end{equation*}
Now denote
\begin{equation}\label{partial neighborhood}
(\partial\Omega)_{R}=\{x\in\Omega\ |\ dist(x,\partial\Omega)\leq R\}.
\end{equation}
By covering $(\partial\Omega)_{R}\times(-R/2,R/2)$ with finite many $\{\Psi(Q_i)\}$, whose number is $CR^{-(n-1)}$, there holds
\begin{equation}\label{boundary nodal set}
\mathcal{H}^n\left\{z\in(\partial\Omega)_{R}\times(-R/2,R/2)\ |\ \bar{u}(z)=0\right\}\leq CA^{n+2}R(\sqrt{\lambda}+1).
\end{equation}
Thus from $\bar{u}(x,x_{n+1})=u(x)e^{\sqrt{\lambda}x_{n+1}}$ and $e^{\sqrt{\lambda}x_{n+1}}\neq0$,
\begin{equation*}
\mathcal{H}^{n-1}\left\{x\in(\partial\Omega)_{R}\ |\ u(x)=0\right\}\leq CA^{n+2}(\sqrt{\lambda}+1).
\end{equation*}

{\it Step 6.}
Now in the $y-$space, we consider another cube $Q^1=[R,3R]^{n}\times[-R,R]$. This is a cube above the cube $Q$ we discussed before. We divide $Q^1$ into $T^{n+1}=(100(n+1))^{n+1}$ equal subcubes. For each subcube, its doubling index is less than or equal to $C(\sqrt{\lambda}+1)$. Then from Lemma \ref{nodal set of cube}, one has
\begin{equation*}
\mathcal{H}^n\left\{y\in Q^1\ |\ \widetilde{u}(y)=0\right\}\leq CT^{n+1}(\sqrt{\lambda}+1)\left(\frac{R}{T}\right)^n\leq CTR^n(\sqrt{\lambda}+1).
\end{equation*}
Using the map $\Psi$  mapping $y$  back to $z$ again,
\begin{equation*}
\mathcal{H}^n\left\{z\in\Psi(Q^1)\ |\ \bar{u}(z)=0\right\}\leq CTR^n(\sqrt{\lambda}+1).
\end{equation*}
By covering $((\partial\Omega)_{3R}\setminus(\partial\Omega)_{R})\times(-R,R)$ with finitely many$\{\Psi(Q^1_{i})\}$,
whose number is $C(2R)^{-(n-1)}$, there holds
\begin{equation*}
\mathcal{H}^n\left\{z\in((\partial\Omega)_{3R}\setminus(\partial\Omega)_{R})\times(-R,R)\ |\ \bar{u}(z)=0\right\}\leq CTR(\sqrt{\lambda}+1).
\end{equation*}
Again from the fact that $\bar{u}(x,x_{n+1})=u(x)e^{\sqrt{\lambda}x_{n+1}}$ and $e^{\sqrt{\lambda}x_{n+1}}\neq0$,
\begin{equation*}
\mathcal{H}^{n-1}\left\{x\in((\partial\Omega)_{3R}\setminus(\partial\Omega)_{R})\ |\ u(x)=0\right\}\leq CT(\sqrt{\lambda}+1).
\end{equation*}
By using the same argument to the cube $Q^2=[3R,7R]^n\times[-2R,2R]$ in $y-$space, and covering $((\partial\Omega)_{7R}\setminus(\partial\Omega)_{3R})\times(-2R,2R)$ with finitely many $\{\Psi(Q^2_{i})\}$, whose number is $C(4R)^{-(n-1)}$,  there also holds
\begin{equation*}
\mathcal{H}^{n-1}\left\{x\in((\partial\Omega)_{7R}\setminus(\partial\Omega)_{3R})\ |\ u(x)=0\right\}\leq CT(\sqrt{\lambda}+1).
\end{equation*}
Similarly, on the $j$-th step ($j>1$), one has
\begin{equation}
\mathcal{H}^{n-1}\left\{x\in\left((\partial\Omega)_{(2^{j+1}-1)R}
\setminus(\partial\Omega)_{(2^{j}-1)R}\right)\ |\ u(x)=0\right\}\leq CT(\sqrt{\lambda}+1).
\end{equation}
Repeat this argument until the side length of the cube $Q^{m}$ is larger than $r_0$. This needs $m-$th step, where
$m\leq C\log_{2}\frac{r_0}{R}$.
Then
\begin{align*}
&\mathcal{H}^{n-1}\left\{x\in(\partial\Omega)_{r_0}\ |\ u(x)=0\right\}\nonumber
\\&=\mathcal{H}^{n-1}\left\{x\in(\partial\Omega)_{R}\ |\ u(x)=0\right\}\nonumber
+\sum\limits_{j=1}^{m}\mathcal{H}^{n-1}\left\{x\in\left((\partial\Omega)_{(2^{j+1}-1)R}
\setminus(\partial\Omega)_{(2^{j+1}-1)R}\right)\ |\ u(x)=0\right\}\nonumber
\\&\leq CA^{n+2}(\sqrt{\lambda}+1)+CTm(\sqrt{\lambda}+1)\nonumber
\\&\leq CA^{n+2}(\sqrt{\lambda}+1)+CT\log_{2}\frac{r_0}{R}(\sqrt{\lambda}+1).
\end{align*}
{\it Step 7.}
Finally from Lemma $\ref{interior nodal set}$,
\begin{equation*}
\mathcal{H}^{n-1}\left\{x\in \Omega_{\frac{r_0}{2}}\ |\ u(x)=0\right\}\leq \frac{C}{r_0}(\sqrt{\lambda}+1).
\end{equation*}
Then we obtain
\begin{align*}
&\mathcal{H}^{n-1}\left\{x\in \Omega\ |\ u(x)=0\right\}\nonumber
\\&\leq \mathcal{H}^{n-1}\left\{x\in \Omega_{\frac{r_0}{2}}\ |\ u(x)=0\right\}+\mathcal{H}^{n-1}\left\{x\in(\partial\Omega)_{r_0}\ |\ u(x)=0\right\}\nonumber
\\&\leq \frac{C}{r_0}(\sqrt{\lambda}+1)+CA^{n+2}(\sqrt{\lambda}+1)+CT\log_{2}\frac{r_0}{R}(\sqrt{\lambda}+1)\\
&\leq C\Big(1+\log\left(\|\nabla V\|_{L^{\infty}(\Omega)}+1\right)\Big)\cdot\left(\|V\|_{L^{\infty}(\Omega)}^{\frac{1}{2}}+|\nabla V\|_{L^{\infty}(\Omega)}^{\frac{1}{2}}+1\right),
\end{align*}
thanks to
$R=R_0\|\nabla V\|_{L^{\infty}}^{-\frac{1}{2}}$, $\lambda=\|V\|_{W^{1,\infty}}$,  $A=A_0$, $R_0$, $r_0$ and $T$ are positive constants depending only on $n$ and $\Omega$.
\qed

\begin{remark}
If
$\|\nabla V\|_{L^{\infty}}$ is small, then all the terms depending on $\|\nabla V\|_{L^{\infty}}$ in the above arguments can be replaced
by a positive constant $C$ depending only on $n$ and $\Omega$.
So in this case, the conclusion becomes
\begin{equation*}
\mathcal{H}^{n-1}\left\{x\in \Omega\ |\ u(x)=0\right\}\leq C\left(\|V\|^{\frac{1}{2}}_{L^{\infty}}+1\right).
\end{equation*}
\end{remark}

\textbf{Funding}
Liu's research is supported by National Natural Science Foundation of China (No. 12471198). Yang's research is supported by National Natural Science Foundation of China (No.12431018).

\textbf{Data availability}
 No data has been generated or analysed during this study.\\

\textbf{Conflict of interest} The authors have no relevant financial or non-financial interests to disclose.

\end{document}